\documentclass[a4paper,11pt,oneside]{article}

\usepackage{amsmath,amsthm,amsfonts,amssymb,amscd}
\usepackage{mathtools}
\usepackage{pgf,tikz}
\usepackage{mathrsfs}
\usepackage{float}
\usepackage{csvsimple}
\usepackage{enumerate}
\usepackage{booktabs}
\usepackage{pgfplots}
\usepackage[noend]{algpseudocode}
\usepackage{graphicx,bm,xcolor}
\usepackage{algorithm}
\usepackage{pdfpages}
\usepackage{algpseudocode}
\usepackage{stmaryrd}
\usepackage[square,sort,comma,numbers]{natbib}
\usetikzlibrary{arrows}
\usepackage{caption}

\usepackage{t1enc}
\usepackage[german,english]{babel}
\usepackage{verbatim} 

\usepackage{url}

\newif\ifMAKEPICS
\MAKEPICSfalse

\ifMAKEPICS
\usepackage[cleanup,subfolder]{gnuplottex}
\usepackage{xparse}

\ExplSyntaxOn
\DeclareExpandableDocumentCommand{\convertlen}{ O{cm} m }
{
	\dim_to_decimal_in_unit:nn { #2 } { 1 #1 } cm
}
\ExplSyntaxOff
\fi

\usepackage{authblk}

\usepackage[top=2cm, bottom=2cm, left=2cm, right=2cm]{geometry}

\theoremstyle{plain} 
\newtheorem{theorem}{Theorem}[section]
\newtheorem{propositon}{Proposition}[section]

\newtheorem{remark}[theorem]{Remark}
\newtheorem{definition}[theorem]{Definition}

\theoremstyle{definition} %

\theoremstyle{remark} %

\newcommand{\ignore}[1]{}

\begin{document}

\title{Multigoal-Oriented Error Estimates for Non-linear Problems} 
	\author[1,2]{B. Endtmayer}
	\author[2]{U. Langer}
	\author[3]{T. Wick}

	\affil[1]{Doctoral Program on Computational Mathematics, Johannes Kepler University, Altenbergerstr. 69, A-4040 Linz, Austria}
	\affil[2]{Johann Radon Institute for Computational and Applied Mathematics, Austrian Academy of Sciences, Altenbergerstr. 69, A-4040 Linz, Austria}
	\affil[3]{Institut f\"ur Angewandte Mathematik, Leibniz Universit\"at Hannover, Welfengarten 1, 30167 Hannover, Germany}

	\date{}

	\maketitle

\begin{abstract}
In this work, we further develop multigoal-oriented a posteriori 
error estimation with two objectives in mind. 
First, 
we formulate goal-oriented mesh
adaptivity for multiple functionals of interest for nonlinear 
problems in which both the Partial Differential Equation (PDE) and the goal 
functionals may be nonlinear. 
%
%
Our method is based  on  a posteriori error estimates
in which the adjoint problem is used and 
a partition-of-unity is employed for the error localization that allows us
to formulate the error estimator in the weak form.
We provide a careful derivation 
of the primal and adjoint parts of the error estimator.
The second objective is concerned with balancing the nonlinear 
iteration error with the discretization error yielding
adaptive stopping rules for Newton's method. 
Our techniques are substantiated with several numerical 
examples including scalar PDEs
and PDE systems,
geometric singularities, and both nonlinear PDEs and 
nonlinear
goal functionals. In these tests, up to six goal functionals 
are simultaneously controlled.
\end{abstract}

\section{Introduction}
\label{Introduction}
A posteriori error estimation and mesh adaptivity 
are well-developed methodologies for finite element 
computations, 
see, e.g., the monographs 
\cite{Verfuerth:1996a,AinsworthOden:2000,BaRa03,NeittaanmaekiRepin:2004a,Han:2005a,RepinBook2008}
and the references therein.
Specifically, 
goal-oriented error estimation is a powerful method 
when 
the evaluation of
certain functionals of interest (often these are technical quantities)
is
the main aim rather 
than 
the computation of
global error norms. Here, the dual-weighted residual (DWR) method 
is often applied
\cite{BeckerRannacher1995,BeRa01}.

Thanks to increasing 
computational resources, multiphysics applications 
such as multiphase flow, porous media applications, fluid-structure 
interaction 
and
electromagnectics 
are currently 
one main focus in applied mathematics and engineering. Here, 
mesh adaptivity (ideally combined with parallel computing) 
can greatly reduce the computational cost while 
measuring functionals of interest with sufficient accuracy. 
Since in multiphysics, several physical phenomena interact, 
it 
might be desirable 
that more than one goal functional shall be controlled. However, 
only a few studies have appeared yet. 
A first methodology was proposed in \cite{HaHou03,Ha08}. 
Other studies can be found in \cite{HouSenSue02,PARDO20101953} and 
more recently in \cite{BruZhuZwie16,EnWi17,KerPruChaLaf2017}. 

Until a few years ago, one principle problem in using the DWR method 
was the fact that the error estimator was based on the strong form of the equations
\cite{BeRa01} or 
the weak form
using special interpolation operators working 
on patched meshes \cite{BraackErn02}. 
In \cite{RiWi15_dwr}, the 
previous localization techniques were analyzed in more detail and additionally
a novel localization strategy based on a partition-of-unity (PU) was proposed.
The PU localization specifically allows for a much simpler application of the DWR method 
to multiphysics and nonlinear, vector-valued equations
\cite{RiWi15_dwr,Wi16_dwr_pff}. In addition, 
the PU-DWR method works well with other discretization techniques 
such as BEM-based FEM \cite{WeiWi18} or the finite cell method \cite{StolfoRademacherSchroeder2017}.
On the other hand, the methodology of the PU-DWR method with multiple 
goal functionals has recently been worked out for linear, scalar-valued problems in \cite{EnWi17}. 

The first goal of this paper is to extend this work 
to nonlinear problems and PDE systems. 
Here, our focus 
is on a careful design of the error estimator that includes 
both
the primal part and the adjoint part. The latter one is often neglected 
in the literature because the evaluation requires additional computational 
cost and renders the method even more expensive. 
It is clear and well-known (see e.g., \cite{BeRa01})
that, in the linear case, the primal and adjoint residuals yield the same error
values, but possibly different locally refined meshes; see
e.g. \cite{RiWi15_dwr}.
In our current work, we will see that the adjoint estimator part is crucial to
obtain good effectivity indices. 
Therefore, this term should not be neglected.

%
	The second objective of this paper is concerned 
with balancing the discretization and the nonlinear iteration error.
In recent years, there has been published some work
on balancing the iteration error (of the linear 
or nonlinear solver) with the discretization error
\cite{ErnVohral2013,BeJohRa95,MeiRaVih109,RanVi2013,RaWeWo10}. 
We base ourselves on \cite{RanVi2013}, and we employ 
specifically the PU localization. 
Consequently,
 the 
DWR method is used to design an adaptive stopping criterion for Newton's
method that is in balance with the estimated discretization error.
The main aspects comprise a careful 
choice of the weighting functions to design 
an appropriate joined goal functional. 
Moreover, we provide all details for the nonlinear solver, which is a Newton-type method 
with backtracking line-search.
	Since we know a solution on the
  previous mesh, we use this solution as initial guess for Newton's method
  yielding a nested iteration. Specifically, nested solution methods or nonlinear nested iterations
were developed, for instance,  in \cite{BeRa01,MultigridHackusch03a}. 
We refer to \cite{MultigridHackusch03a, Reusken1988}
for the analysis of nested iteration methods.

In summary, the goals of this work are two-fold:
\begin{itemize}
\item Design of the PU-DWR method for multigoal-oriented error estimation
for nonlinear problems and PDE systems.
\item Balancing iteration and discretization errors for nonlinear multigoal-oriented
error estimation and mesh adaptivity. The nonlinearities 
may appear in the PDE itself as well as in the goal functionals.
\end{itemize}

The outline of this is paper is as follows:
In Section \ref{Discretization}, our setting is described. 
Next, in Section~\ref{PU-DWR-NONLINEAR-ONEFUNTIONAL}, we describe 
the methodology for one goal functional. This is followed 
by a detailed derivation of a multigoal-oriented approach 
presented in Section \ref{Multigoalfunctionals}. The key 
algorithms are formulated in Section~\ref{sec_alg}. 
In Section \ref{sec_num_tests} several numerical tests substantiate 
our developments. We summarize our work in Section~\ref{sec_concl}.









\section{An abstract setting}
\label{Discretization}
%
%
Let $U$ and $V$ be Banach spaces, and let $\mathcal{A}: U \mapsto V^*$ be a
(possibly) nonlinear 
operator,
where
$V^*$ denotes the dual space
of the Banach space $V$. 
We have in mind nonlinear differential operators $\mathcal{A}$
acting between Sobolev spaces. 
We now consider the following weak formulation of the operator 
equation $\mathcal{A}(u) = 0$ in $V^*$:
Find $u \in U$ such that 
\begin{equation}
\label{Problem:primal}
\mathcal{A}(u)(v)=0 \quad \forall v \in V.
\end{equation}
The discretization of the nonlinear variational problem (\ref{Problem:primal})
can be performed by means of different methods. 
Our favored method is the  Finite Element Method (FEM), 
see also Section~\ref{subsec_SpatialDiscretization}.
The corresponding discrete problem reads as follows:
Find $u_h \in U_h$ such that 
\begin{equation} \label{problem: discrete primal problem}
\mathcal{A}(u_h)(v_h)=0 \quad \forall v_h \in V_h,
\end{equation}
where $U_h$ and  $V_h$ are finite-dimensional subspaces of $U$ and $V$,
respectively.
For the time being, let us assume that both problems are solvable.
Later we will specify our assumptions imposed on $\mathcal{A}$.
We are primarily not interested in approximating a solution $u$
of (\ref{Problem:primal}), but in the approximate computation of one or more possibly nonlinear functionals at a solution.

An example for such an operator $\mathcal{A}$ is given by the weak formulation
of the regularized $p$-Laplace equation 
(see also \cite{DiRu07,Hi2013,ToWi2017})  
that reads as follows:
Find $u \in U := W^{1,p}_0(\Omega)$ such that
\begin{align}
\label{pLaplace:weak}
  \mathcal{A}(u)(v) :=& \langle{(\varepsilon^2+ |\nabla u|^2)^{\frac{p-2}{2}}\nabla u,\nabla v\rangle}_{(L^p(\Omega))^*\times L^p(\Omega)}-\langle{f,v\rangle}_{ (W^{1,p}_0(\Omega))^*\times W^{1,p}_0(\Omega)}
  = 0
\end{align}
for all $v \in  V := W^{1,p}_0(\Omega)$,
where $\varepsilon$ denotes  a fixed positive regularization parameter,
$f \in (W^{1,p}_0(\Omega))^*= W^{-1,q}(\Omega)$ is some given source,
with $p^{-1} + q^{-1}=1$ and fixed $p>1$,
and $\langle{\cdot,\cdot}\rangle$ denots the corresponding duality products.
Here, $\Omega \subset \mathbb{R}^d$, $d=1,2,3$, is a bounded Lipschitz domain, 
and  $W^{1,p}_0(\Omega)$ denotes the usual Sobolev space of all functions
from the Lebesgue space $L^p(\Omega)$ with weak derivatives in $L^p(\Omega)$
and trace zero on the boundary $\partial \Omega$,
see, e.g., \cite{Adams2003sobolev}.
The notation $ | \cdot | $ is used for the Euclidean norm of some vector.
The corresponding strong form is formally given by
\begin{align*}
  -\text{div}((\varepsilon^2 + |\nabla u|^2)^{\frac{p-2}{2}}\nabla u) &=f  \quad  \mbox{in}\;\Omega, \\
  u &= 0 \quad  \mbox{on}\; \partial \Omega .
\end{align*}
In Subsection~\ref{subsection: Example1}, 
the regularized $p$-Laplace~(\ref{pLaplace:weak}) serves as first example 
for our numerical experiments.
\begin{remark}
We refer the reader to \cite{Glowinski1975} for the investigation
of the original $p$-Laplace problem.
\end{remark}

 \section{The dual weighted residual method for nonlinear problems 
 in the case of a single-goal functional} 
\label{PU-DWR-NONLINEAR-ONEFUNTIONAL}
In this section, we apply the DWR method to nonlinear
problems. The general method was developed in \cite{BeRa01}. 
The extension to balance discretization and iteration errors was undertaken 
in \cite{MeiRaVih109,RaWeWo10,RanVi2013}. We base ourselves 
on the latter study \cite{RanVi2013}, in which algorithms for 
nonlinear problems have been worked out. This last paper, together 
with \cite{RiWi15_dwr,EnWi17}, form the basis of the current paper.
We are interested in the goal functional evaluation $J:U\to\mathbb{R}$
with $u\mapsto J(u)$, where $u\in U$ is a solution of the primal problem (\ref{Problem:primal}).
Examples for such goal functionals are:
\begin{itemize}
	\item point evaluation: $$	J(u):= u(x_0), $$
	\item integral evaluation: $$  J(u):= \int_{\Omega}^{}u(x)\xi(x)\,dx, $$
	\item nonlinear functional evaluation: $$  J(u):= \int_{\Omega}^{}u(x)\xi(x)u(x_0)^2 \,dx\int_{\Omega}^{}u(y)\phi(y)\,dy, $$
\end{itemize}
where $\xi$ and $\phi$ are given functions from $L^2(\Omega)$ and $x_0$ a given point in $\Omega$.
For the DWR approach we need to solve the adjoint problem:
Find $z \in V$ corresponding to $u \in U$ such that
\begin{equation}\label{Problem: adjoint}
\mathcal{A}'(u)(v,z)=J'(u)(v) \quad \forall v \in U,
\end{equation}
where $u$ denotes 
a (primal) solution of the primal problem (\ref{Problem:primal}), and $\mathcal{A'}(u)$ and $J'(u)$ denote the Fr\'echet-derivatives of the nonlinear operator or functional, respectively, evaluated at $u$.
Later we also need the corresponding discrete solution of the adjoint problem. 
This reads as follows: Find $z_h \in V_h$ corresponding to $u_h \in U_h$ such that 
	\begin{equation}\label{Problem: discrete adjoint on small space}
	\mathcal{A}'(u_h)(v_h,z_h)=J'(u_h)(v_h) \quad \forall v_h \in U_h,
	\end{equation}
	with $u_h$ as a solution of (\ref{problem: discrete primal problem}).

Similarly to  the findings in \cite{RanVi2013,BeRa01,RaWeWo10} for the Galerkin
case ($U=V$),  we derive  an error representation in the following theorem:
\begin{theorem}\label{Theorem: Error Representation}
	Let us assume that $\mathcal{A} \in \mathcal{C}^3(U,V)$ and $J \in \mathcal{C}^3(U,\mathbb{R})$. 
	If $u$ solves (\ref{Problem:primal}) and $z$ solves (\ref{Problem: adjoint})
	for $u \in U$, 
	then it holds for arbitrary fixed  $\tilde{u} \in U$ and $ \tilde{z} \in V$ :
	\begin{align} \label{Error Representation}
		\begin{split}
		J(u)-J(\tilde{u})&= \frac{1}{2}\rho(\tilde{u})(z-\tilde{z})+\frac{1}{2}\rho^*(\tilde{u},\tilde{z})(u-\tilde{u}) 
		-\rho (\tilde{u})(\tilde{z}) + \mathcal{R}^{(3)},
		\end{split}
	\end{align}
	 where
	\begin{align}
		\label{Error Estimator: primal}
		\rho(\tilde{u})(\cdot) &:= -\mathcal{A}(\tilde{u})(\cdot), \\
		\label{Error Estimator: adjoint}
		\rho^*(\tilde{u},\tilde{z})(\cdot) &:= J'(u)-\mathcal{A}'(\tilde{u})(\cdot,\tilde{z}), 	
	\end{align}
	and the remainder term
	\begin{equation}
			\begin{split}	\label{Error Estimator: Remainderterm}
			\mathcal{R}^{(3)}:=\frac{1}{2}\int_{0}^{1}[J'''(\tilde{u}+se)(e,e,e)
			-\mathcal{A}'''(\tilde{u}+se)(e,e,e,\tilde{z}+se^*)
			-3\mathcal{A}''(\tilde{u}+se)(e,e,e)]s(s-1)\,ds,
			\end{split} 
	\end{equation}
with $e=u-\tilde{u}$ and $e^* =z-\tilde{z}$.
\end{theorem}
\begin{proof}
	For the completeness of the presentation we add the proof below, which is very similar to \cite{RanVi2013}.
	First we define $x := (u,z) \in  X:=U \times V$ and $\tilde{x}:=(\tilde{u},\tilde{v}) \in X$.
	By assuming that $\mathcal{A} \in \mathcal{C}^3(U,V)$ and $J \in
	\mathcal{C}^3(U,\mathbb{R})$ we know that the Lagrange function 
	\begin{equation*}
	\mathcal{L}(\hat{x}):= J(\hat{u})-\mathcal{A}(\hat{u})(\hat{z}) \quad \forall (\hat{u},\hat{z})=:\hat{x} \in X,
	\end{equation*}
	is in $\mathcal{C}^3(X,\mathbb{R})$. Assuming this it holds
	\begin{equation*}
	\mathcal{L}(x)-\mathcal{L}(\tilde{x})=\int_{0}^{1} \mathcal{L}'(\tilde{x}+s(x-\tilde{x}))(x-\tilde{x})\,ds.
	\end{equation*}
	Using  the trapezoidal rule \cite{RanVi2013}, we obtain
	\begin{equation*}
	\int_{0}^{1}f(s)\,ds =\frac{1}{2}(f(0)+f(1))+ \frac{1}{2} \int_{0}^{1}f''(s)s(s-1)\,ds,
	\end{equation*}
	for $f(s):= \mathcal{L}'(\tilde{x}+s(x-\tilde{x}))(x-\tilde{x})$ we conclude
	\begin{align*}
	\mathcal{L}(x)-\mathcal{L}(\tilde{x}) =& \frac{1}{2}(\mathcal{L}'(x)(x-\tilde{x}) +\mathcal{L}'(\tilde{x})(x-\tilde{x})) + \mathcal{R}^{(3)}.
	\end{align*}
	From the definition of $\mathcal{L}$ we observe that
		\begin{align*}
		J(u)-J(\tilde{u})=\mathcal{L}(x)-\mathcal{L}(\tilde{x}) +\underbrace{A(u)(z) }_{=0} + A(\tilde{u})(\tilde{z}) =& \frac{1}{2}(\mathcal{L}'(x)(x-\tilde{x}) +\mathcal{L}'(\tilde{x})(x-\tilde{x})) +A(\tilde{u})(\tilde{z})+ \mathcal{R}^{(3)}.
		\end{align*}
		It remains to show that $\frac{1}{2}(\mathcal{L}'(x)(x-\tilde{x}) +\mathcal{L}'(\tilde{x})(x-\tilde{x})) = \frac{1}{2}\rho(\tilde{u})(z-\tilde{z})+\frac{1}{2}\rho^*(\tilde{u},\tilde{z})(u-\tilde{u}) $.
		But this is true since
		\begin{align*}
		\mathcal{L}'(x)(x-\tilde{x}) +\mathcal{L}'(\tilde{x})(x-\tilde{x}) = & \underbrace{J'(u)(e)-\mathcal{A}'(u)(e,z)}_{=0}-\underbrace{A(u)(e^*)}_{=0}+\underbrace{J'(\tilde{u})(e)-\mathcal{A}'(\tilde{u})(e,\tilde{z})}_{=\rho^*(\tilde{u},\tilde{z})(u-\tilde{u})}-\underbrace{A(\tilde{u})(e^*)}_{=-\rho(\tilde{u})(z-\tilde{z})}.
		\end{align*}
\end{proof}

	\begin{remark}
	Instead of $\mathcal{A} \in \mathcal{C}^3(U,V)$ and $J \in \mathcal{C}^3(U,\mathbb{R})$ it is sufficient that $\mathcal{A} \in \mathcal{C}^2(U,V)$, $J \in \mathcal{C}^2(U,\mathbb{R})$ and $J''',$ $ \mathcal{A'''}$ exist and  are bounded.
		Moreover one can further relax these assumptions. Indeed the boundedness of the derivatives is just needed in the set $\{ w \in U| w=(1-s)u+s\tilde{u}  \}$ and just in direction $u-\tilde{u}$.
	\end{remark}
	\begin{remark}
		It might happen that $\mathcal{A} \in \mathcal{C}^3(U,V)$ and $J \in \mathcal{C}^3(U,\mathbb{R})$ do not hold for the continuous spaces.
		Since the result holds for general Banach spaces $U$ and $V$,
                it is sufficient to be shown for the discrete spaces $U_{h,u},V_{h,z}$, where $U_{h,u}:=\{w+cu| w \in U_h, c \in \mathbb{R}\},V_{h,z}:=\{v+cz|v\in V_h, c \in \mathbb{R}\}$. 
	\end{remark}

\begin{remark}
	In accordance with the literature, we  denote the parts
	$\rho(\tilde{u})(z-\tilde{z})$ and
	$\rho^*(\tilde{u},\tilde{z})(u-\tilde{u})$
	by \textit{primal error
		estimator} and \textit{adjoint error estimator},
	respectively. 
The remainder term $\mathcal{R}^{(3)}$, as in (\ref{Error Estimator:
  Remainderterm}), is of the order $\mathcal{O}(\Vert e \Vert_U^2
\text{max}(\Vert e \Vert_U, \Vert e^* \Vert_V))$. Therefore, it can be
neglected if $\{\tilde{u},\tilde{z}\}$ are close enough to $\{u,z\}$.  
\end{remark}

As in \cite{RanVi2013} ,we can identify

 		\begin{equation} \label{Error Estimator: discretization}
 		\eta_h:=|\frac{1}{2}\rho(\tilde{u})(z-\tilde{z})+\frac{1}{2}\rho^*(\tilde{u},\tilde{z})(u-\tilde{u})|,
 		\end{equation}
 		as the \textit{discretization error} and \begin{equation}
 		\eta_m:=|\rho(\tilde{u})(\tilde{z})|,
 		\end{equation} as the \textit{linearization error} if we neglect the remainder term $\mathcal{R}^{(3)}$.
 		Since Theorem \ref{Theorem: Error Representation} is valid for arbitrary $\tilde{z}$ and $\tilde{u}$ it also holds for approximations $u_h$ and $z_h$, even if they are not computed exactly.
 		Of course, formula (\ref{Error Estimator: discretization}) still contains an exact solution $u$. Since $u$ is not known, we either use an approximation in an enriched discrete space (for example, in a finite element space, with higher polynomial degree), or we use an  interpolant $I_h^{h_2}$, such as  in \cite{BeRa01}, to obtain a more accurate solution $u_h^{(2)}$. If not mentioned otherwise, we use the approximation in the enriched (finite element) space. An enriched discrete space is also used to compute an approximation $z_h^{(2)}$ of $z$. 
If one would use the same finite-dimensional space as for the test space used in the 
discrete primal problem (\ref{problem: discrete primal problem}), then $\mathcal{A}(u_h)(z_h)=0 $ for our approximate solution $u_h$ of (\ref{problem: discrete primal problem}) (if the nonlinear problem is solved exactly).
 Therefore, the discrete
                 adjoint problem reads as follows: 
 		Find $z_h^{(2)}\in V_h^{(2)}$ such that
 		\begin{equation}\label{Problem: discrete adjoint}
 		\mathcal{A}'(u_h^{(2)})(v_h^{(2)},z_h^{(2)})=J'(u_h^{(2)})(v_h^{(2)})\quad \forall v_h^{(2)} \in U_h^{(2)},
 		\end{equation}
 		where $U_h^{(2)}$  and $V_h^{(2)}$ denote the enriched finite
                dimensional spaces, and $u_h^{(2)}$ denotes the more accurate
                solution, obtained by solving (\ref{problem: discrete primal
                  problem}) with $U_h=U_h^{(2)}$ and $V_h=V_h^{(2)}$ or by
                interpolation $u_h^{(2)}=I_h^{h_2}u_h$. 
With these approximations, the practical error estimator reads:
 		\begin{equation} \label{Error Estimator: discretization inexact}
 		\eta_h:=|\frac{1}{2}\rho(u_h)(z_h^{(2)}-z_h)+\frac{1}{2}\rho^*(u_h,z_h)(u_h^{(2)}-u_h)|.
 		\end{equation}
 	For localization of the error estimator, we use the partition of unity (PU) technique which is presented in \cite{RiWi15_dwr}. This means that we choose a set of functions $\{\psi_1, \psi_2,  \cdots,\psi_N\}$ such that $	\sum_{i=1}^{N} \psi_i \equiv 1.$
 		Inserting this into (\ref{Error Estimator: discretization inexact}) leads to
 		\begin{equation} \label{Local Error Estimator: discretization}
 		\eta_{h}:=|\sum_{i=1}^{N}\eta_i|,
 		\end{equation}
 		with 
                \begin{equation}
                  \label{eta_i_PU}
                  \eta_i:=\frac{1}{2}\rho(\tilde{u})((z_h^{(2)}-\tilde{z})\psi_i)+\frac{1}{2}\rho^*(\tilde{u},\tilde{z})((u_h^{(2)}-\tilde{u})\psi_i).
                \end{equation}
                We notice that in the primal part of the error
                  indicator $\tilde{z}$ is replaced by $i_h z_h^{(2)}$ as in
                  \cite{BeRa01}. 
 		For instance, a typical partition of unity is given by the finite element basis.
 		  In this case, we distribute $|\eta_i|$ to the corresponding elements with a certain weight as for example illustrated in Figure~\ref{fig: dist. of nodal error}.
		\begin{figure}[H]
 			\centering
 			\scalebox{0.65}{
	 			\definecolor{qqzzff}{rgb}{0,0.6,1}\definecolor{ffqqtt}{rgb}{1,0,0.2}\definecolor{zzttqq}{rgb}{0.6,0.2,0}\begin{tikzpicture}[line cap=round,line join=round,>=triangle 45,x=1cm,y=1cm]\clip(-4.2823850829861,3.7602936957858014) rectangle (4.23939348405163,12.502011747831432);\fill[line width=2pt,color=zzttqq,fill=zzttqq,fill opacity=0.05000000149011612] (-3.90437811624889,4.040980807321905) -- (4.09562188375111,4.040980807321905) -- (4.095621883751111,12.040980807321905) -- (-3.904378116248889,12.040980807321905) -- cycle;\draw [line width=2pt,color=zzttqq] (4.09562188375111,4.040980807321905)-- (4.095621883751111,12.040980807321905);\draw [line width=2pt,color=zzttqq] (4.095621883751111,12.040980807321905)-- (-3.904378116248889,12.040980807321905);\draw [line width=2pt,color=zzttqq] (4.095621883751111,4.040980807321905)-- (-3.904378116248889,4.040980807321905);\draw [line width=2pt,color=zzttqq] (-3.904378116248889,12.040980807321905)-- (-3.90437811624889,4.040980807321905);\draw [line width=2pt] (-3.
9043781162488895,8.040980807321905)-- (4.095621883751111,8.040980807321905);\draw [line width=2pt] (0,12)-- (0.09562188375111003,4.040980807321905);\draw [->,line width=2pt,color=qqzzff] (0.04756476443214799,8.040980807321905) -- (-2,10);\draw [->,line width=2pt,color=qqzzff] (0.04756476443214799,8.040980807321905) -- (2,10);\draw [->,line width=2pt,color=qqzzff] (0.04756476443214799,8.040980807321905) -- (2,6);\draw [->,line width=2pt,color=qqzzff] (0.04756476443214799,8.040980807321905) -- (-2,6);\begin{scriptsize}\draw [fill=ffqqtt] (0.04756476443214799,8.040980807321905) circle (4.5pt);\draw[color=ffqqtt] (-0.7388539034218957,8.3226237278874394) node {\normalsize$|\eta_i|$};\draw[color=qqzzff] (-0.8730578512396703,9.545853425312479) node {$\frac{1}{4}|\eta_i|$};\draw[color=qqzzff] (1.489642783962842,8.8109250734239465) node {$\frac{1}{4}|\eta_i|$};\draw[color=qqzzff] (0.5344757871730055,6.964588484393144) node {$\frac{1}{4}|\eta_i|$};\draw[color=qqzzff] (-1.1462632333856986,7.360736288504885) node {$\frac{1}{4}|\eta_i|$};\end{scriptsize}\end{tikzpicture}
 			}
 			\caption{Equal distribution of the local error estimator using the  $Q^1_c$ basis function at the central vertex to the corresponding elements as in \cite{RiWi15_dwr}, see also Section~\ref{subsec_SpatialDiscretization}.}\label{fig: dist. of nodal error}
 		\end{figure}
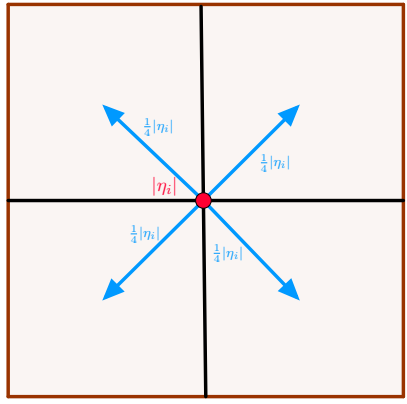

\section{Multiple-goal functionals}
\label{Multigoalfunctionals}
Now let us assume that we are interested in the evaluation of $N$ functionals, which we denote by
$J_1, J_2, \ldots,J_N$.
From Section \ref{PU-DWR-NONLINEAR-ONEFUNTIONAL}, we know how to compute a local error estimator for one functional.
We could compute the local error estimators separately. However, we would
have to solve the adjoint problem (\ref{Problem: adjoint}) $N$ times
\cite{HaHou03,Ha08}.
Let us now  assume that a solution $u$ of problem (\ref{Problem:primal}) and the chosen $\tilde{u} \in U$ belong to $\bigcap_{i=1}^N \mathcal{D}(J_i)$, where $\mathcal{D}(J_i)$ describes the domain of $J_i$. 
\begin{definition}[error-weighting function]
	Let $ M \subseteq \mathbb{R}^N$. 
We say that $\mathfrak{E}: (\mathbb{R}^+_0)^N \times  M \mapsto \mathbb{R}^+_0$ is an \textit{error-weighting
          function} if  $\mathfrak{E}(\cdot,m) \in
        \mathcal{C}^1((\mathbb{R}^+_0)^N,\mathbb{R}^+_0)$ is strictly monotonically
        increasing in each component and $\mathfrak{E} (0,m)=0$ for all
        $m \in M$.
\end{definition}
Let $\vec{J}: \bigcap_{i=1}^N \mathcal{D}(J_i) \subseteq U \mapsto \mathbb{R}^N$ be defined as $\vec{J}(v):=(J_1(v),J_2(v),\cdots, J_{N}(v) )$ for all $v \in \bigcap_{i=1}^N \mathcal{D}(J_i)$. Furthermore, we define the operation $|\cdot|_N:\mathbb{R}^N\mapsto (\mathbb{R}^+_0)^N$ as $|x|_N:= (|x_1|,|x_2|,\cdots,|x_N|)$ for $x \in \mathbb{R}^N $.
This allows us to define the \text{error functional} as follows
\begin{align}\label{ErrorrepresentationFunctional}
\tilde{J}_{\mathfrak{E}}(v):=\mathfrak{E}(|\vec{J}(u)-\vec{J}(v)|_N, \vec{J}(\tilde{u})) \qquad \forall v \in \bigcap_{i=1}^N \mathcal{D}(J_i).
\end{align}
It is trivial to see from  the definition of $\mathfrak{E}$ that $J_\mathfrak{E}(v) \in \mathbb{R}^+_0$ for all $v \in \bigcap_{i=1}^N \mathcal{D}(J_i)$.

\begin{remark}
The idea of the construction of $\tilde{J}_{\mathfrak{E}}(v)$ is that
$\mathfrak{E}(|\vec{J}(u)-\vec{J}(v)|_N, \vec{J}(\tilde{u}))$ is a semi-metric (as in \cite{SierpinskiTopo52,KhaKirMetric2001}) on
the set of equivalence classes
$(\vec{J})^{-1}(\mathcal{R}(\vec{J})):=\{(\vec{J})^{-1}(x): x \in
\mathcal{R}(\vec{J})\}$, where $(\vec{J})^{-1}(x):=\{v \in  \bigcap_{i=1}^N
\mathcal{D}(J_i): \vec{J}(v)=x \}$, with $\mathcal{R}(\vec{J})$ denotes the range
of $\vec{J}$, measuring the distance between the equivalence classes
containing $u$ and $v$. Hence, $\tilde{J}_{\mathfrak{E}}(v)$ represents a
semi-metric distance which ensures that $\tilde{J}_{\mathfrak{E}}$  is
monotonically increasing if $|J_i(u)-J_i(\tilde{u})|$ is monotonically
increasing.	
\end{remark}

\begin{remark}
	If we drop the monotonicity condition in the definition of $\mathfrak{E}$, then, for example, $$\mathfrak{E}(|\vec{J}(u)-\vec{J}(v)|_N, \vec{J}(\tilde{u})):= \prod_{i=0}^{N} |J_i(u)-J_i(v)|,$$ would be an error-weighting function, resulting in $J_\mathfrak{E}(\tilde{u})=0$ iff $J_i(u)=J_i(\tilde{u})$ at least for one $i \in \{1,2, \cdots ,N\} $. \end{remark}

\begin{remark}
	The derivation given in this section holds for a general $\tilde{u}$ such that $\vec{J}(\tilde{u}) \in M$. In particular, we are interested in $\tilde{u}$ to be an approximation to $u_h$ solving (\ref{problem: discrete primal problem}).
\end{remark}
The weak derivative of (\ref{ErrorrepresentationFunctional}) in $U$  at $\tilde{u}$ is given by
\begin{align}\label{ErrorrepresentationFunctionalJ'}
\tilde{J}_{\mathfrak{E}}'(\tilde{u})(v):=-\sum_{i=1}^{N}\text{ sign}(J_i(u)-J_i(\tilde{u}))\frac{\partial \mathfrak{E}}{\partial x_i}(|\vec{J}(u)-\vec{J}(\tilde{u})|_N, \vec{J}(\tilde{u})) J_i'(\tilde{u})(v)  \qquad \forall v \in \mathcal{D}(\tilde{J}_{\mathfrak{E}}'(\tilde{u})),
\end{align}
with 

\begin{equation} \label{sign}
\text{sign}(x):=\begin{cases}
	\frac{x}{|x|},	\quad \text{for }x\not =0 ,\\
	0 \quad else
	\end{cases} 
\end{equation}

In \cite{HaHou03,Ha08,EnWi17}, the functionals where combined as follows
\begin{equation}\label{J_c}
\tilde{J}_c(v):=\sum_{i=1}^{N} \underbrace{\frac{\omega_i\text{ sign}(J_i(u)-J_i(\tilde{u}))}{|J_i(\tilde{u})|}}_{=:w_i }J_i(v) \quad \forall v\in \bigcap_{i=0}^N \mathcal{D}(J_i).
\end{equation}

%

Carefully inspecting \cite{Ha08}, we see that the following result can be established:
\begin{propositon}\label{JeJc}
	If $\tilde{J}_c$ is defined as in (\ref{J_c}) and $\tilde{J}_{\mathfrak{E}}$ as in (\ref{ErrorrepresentationFunctional}), then we have
	\begin{align}
	\tilde{J}_c(u)-\tilde{J}_c(\tilde{u})&= \tilde{J}_{\mathfrak{E}}(\tilde{u}),  \label{Jcu-Jcuh=Jeuh}\\
	-\tilde{J}_c'(\tilde{u})(v)&= \tilde{J}_{\mathfrak{E}}'(\tilde{u})(v), \qquad \forall v\in \mathcal{D}(\tilde{J}_{c}'(\tilde{u}))\cap\mathcal{D}(\tilde{J}_{\mathfrak{E}}'(\tilde{u})) ,\label{Jc'uh=Je'uh}\\
	 \mathcal{D}(\tilde{J}_{c}'(\tilde{u}))&=\mathcal{D}(\tilde{J}_{\mathfrak{E}}'(\tilde{u})) \label{DJe=DJc}
	\end{align}
	with $\mathfrak{E}(x,\vec{J}(\tilde{u})):= \sum_{i=1}^{N} \frac{\omega_i x_i}{|J_i(\tilde{u})|}$.
	\begin{proof}[{Proof}]
		First we conclude that 
		\begin{align*}
		\tilde{J}_c(u)-\tilde{J}_c(\tilde{u}) &= \sum_{i=1}^{N} \frac{\omega_i\text{ sign}(J_i(u)-J_i(\tilde{u}))}{|J_i(\tilde{u})|}(J_i(u)-J_i(\tilde{u}))\\
		&= \sum_{i=1}^{N}
                \frac{\omega_i|J_i(u)-J_i(\tilde{u})|}{|J_i(\tilde{u})|}\\ 
                &=
                \mathfrak{E}(|\vec{J}(u)-\vec{J}(\tilde{u})|_N,\vec{J}(\tilde{u})) 
                =\tilde{J}_{\mathfrak{E}}(\tilde{u}),
		\end{align*}
		which already shows (\ref{Jcu-Jcuh=Jeuh}).
		The weak derivative of $\tilde{J}_c$ is given by
		\begin{equation} \label{J_c'}
		\tilde{J}_c'(\tilde{u})(v)=\sum_{i=1}^{N} \frac{\omega_i\text{ sign}(J_i(u)-J_i(\tilde{u}))}{|J_i(\tilde{u})|}J_i'(\tilde{u})(v).
		\end{equation}
		From $\frac{\partial \mathfrak{E}}{\partial x_i}(|\vec{J}(u)-\vec{J}(\tilde{u})|_N, \vec{J}(\tilde{u}))= \frac{\omega_i}{|J_i(\tilde{u})|}$ for all $i \in \{1,2,\cdots, N \}$, and because (\ref{J_c'}) and (\ref{ErrorrepresentationFunctionalJ'}) coincide up to the sign, it holds that (\ref{Jc'uh=Je'uh}) and (\ref{DJe=DJc}) are valid.
	\end{proof}
\end{propositon}
\begin{remark} 
	$\mathfrak{E}(x,\vec{J}(\tilde{u})):= \sum_{i=1}^{N} \frac{\omega_i
    x_i}{|J_i(\tilde{u})|}$ 
is an error-weighting function with $M:=\{x \in \mathbb{R}^N: \min(|x|)>0\}$ provided that  $\omega_i > 0$ for all $i=1,2, \ldots, N$.
\end{remark}
\begin{remark}
	Proposition \ref{JeJc} does not use the property that $u$ solves (\ref{Problem:primal}). We just need that $u \in \bigcap_{i=0}^N \mathcal{D}(J_i)$. However, the goal is to measure the error to an exact solution.
\end{remark}
Since an exact solution $u$ is not known, neither $\tilde{J}_c$ nor $\tilde{J}_{\mathfrak{E}}$ can be constructed. As in Section \ref{PU-DWR-NONLINEAR-ONEFUNTIONAL}, we use the approximation $u_h^{(2)}$ instead of an exact solution $u$  to approximate $\tilde{J}_c$ or $\tilde{J}_{\mathfrak{E}}$ by $J_c$ and $J_{\mathfrak{E}}$, respectively. This approximation reads as follows
\begin{align}\label{ErrorrepresentationFunctionalapprox}
J_{\mathfrak{E}}(v):=\mathfrak{E}(|\vec{J}(u_h^{(2)})-\vec{J}(v)|_N, \vec{J}(\tilde{u})) \qquad \forall v \in \bigcap_{i=1}^N \mathcal{D}(J_i),
\end{align}
with the derivative
\begin{align}\label{ErrorrepresentationFunctionalapprox'}
J_{\mathfrak{E}}'(\tilde{u})(v):=-\sum_{i=1}^{N}\text{ sign}(J_i(u_h^{(2)})-J_i(\tilde{u}))\frac{\partial \mathfrak{E}}{\partial x_i}(|\vec{J}(u_h^{(2)})-\vec{J}(\tilde{u})|_N, \vec{J}(\tilde{u})) J_i'(\tilde{u})(v)  \qquad \forall v \in \mathcal{D}(\tilde{J}_{\mathfrak{E}}'(\tilde{u})).
\end{align}
Using this approximation of the error functional, we can apply the methods for the single-functional case in Section \ref{PU-DWR-NONLINEAR-ONEFUNTIONAL} with $J=J_{\mathfrak{E}}$. 
\begin{remark}
	We notice that Theorem \ref{Theorem: Error Representation}
 formally does not hold for $\tilde{J}_{\mathfrak{E}}$ since the sign-function enters. However, if
 $\mathfrak{E}(\cdot,m) \in \mathcal{C}^3((\mathbb{R}^+_0)^N,\mathbb{R}^+_0)$  and the functionals are sufficiently smooth, then the
 singularities (due to the signum function) in higher derivatives of
  $J_{\mathfrak{E}}$ just appear if $J_i(u)=J_i(u_h)$, or more precisely
 $J_i(u_h^{(2)})=J_i(u_h)$, 
since we use the better approximation $u_h^{(2)}$ instead of $u$.
Alternatively, we can replace the signum function with a sufficiently smooth
approximation. 
\end{remark}

\section{Algorithms}
\label{sec_alg}
We now describe the algorithmic realization of the previous 
methods when we use the FEM as spatial discretization.
To this end,
we first introduce the finite element (FE) discretizations 
that we are going to use 
in our numerical experiments presented in Section~\ref{sec_num_tests}.
Then we
recapitulate the basic structure of Newton's method including 
a line search procedure. Afterwards, we state the 
adaptive Newton algorithm for multiple-goal functionals followed 
by the structure of the final algorithm.


\subsection{Spatial discretization} 
\label{subsec_SpatialDiscretization}

For simplicity, we assume that $\Omega \subset \mathbb{R}^d$ is a polyhedral domain.
Let $\mathcal{T}_h$ be a subdivision (trianglation) of $\Omega$ into quadrilateral elements 
such that $\bigcup_{K \in \mathcal{T}_h}\overline{K}=\overline{\Omega}$
and $K \cap K'= \emptyset$ for all $K,K' \in \mathcal{T}_h$ with $K \neq K'$.
Furthermore, let $\psi_K$ be a multilinear  mapping from the reference domain $\hat{K}=(0,1)^d$ to the element $K \in \mathcal{T}_h$.
We now define the space $Q_c^r$  as
\begin{equation}\label{FE-space-part}
	 Q_c^{r}:=\{ v_h \in C(\overline{\Omega}): v_{h|K} \in Q_r(K), \,   \forall K \in \mathcal{T}_h\},
\end{equation}
with $Q_r(K):=\{v_{|\hat{K}}\circ\psi_K^{-1}:\, v(\hat{x})= \prod_{i=1}^{d} (\sum_{\beta=0}^{r} c_{\beta,i}\hat{x}_i^{\beta}),\, c_{\beta,i} \in \mathbb{R}\}$.
Specifically, we use continuous 
tensor-product finite elements as described in
\cite{Ciarlet:2002:FEM:581834} and  \cite{Braess}.
We also refer the reader to \cite{ArnBofFal2002} 
for the specific approximation properties of 
these finite element spaces.
Let $\mathcal{T}_h^l$ be the triangulation of refinement level $l$. 
Then our finite element spaces are given by $U_h^l:= U \cap Q_c^r$ and $V_h^l:= V \cap Q_c^r$,
whereas the enriched finite element spaces 
are defined by  $U_h^{l,(2)}:= U \cap Q_c^{\tilde{r}}$ and  $V_h^{l,(2)}:= V \cap Q_c^{\tilde{r}}$, 
where
$Q_c^r$ and $Q_c^{\tilde{r}}$ are defined as in (\ref{FE-space-part}) with $\mathcal{T}_h= \mathcal{T}_h^l$ and $\tilde{r} > r$. 
If $U$ and $V$ are spaces of vector-valued functions, 
then intersection has to be understood component-wise with possibly different $r$ in each component.

	\begin{remark}
		The algorithms presented in this section are formulated for FEM \cite{CaOd84,Braess,ArnBofFal2002,Ciarlet:2002:FEM:581834}. 
		 However, we are 
		 not restricted to a particular discretization technique, 
		 but 
		 we must be able 
		 to realize the adaptivity in an appropriate way. 
		 For instance, in isogeometric analysis (IGA) that was originally introduced in 
		 \cite{HughesCottrellBazilevs:2005a}
		 on tensor-product meshes, local mesh refinement is more challenging
		 than in the FEM. Truncated hierarchical B-splines (THB-splines) 
		 are one possible choice to create localized basises 
		 which form a PU, see \cite{GiJuHeTHBsplines2012}.
		 
		 Higher-order B-splines of highest smoothness even on coaerser 
		 meshes can be used to construct enriched 
		 spaces $U_h^{l,(2)}$ and $V_h^{l,(2)}$ that lead to cheap 
		 problems on the enriched spaces, see 
\cite{KleissTomar:2015a,LangerMatculevichRepin:2017a,LangerMatculevichRepin:2017b}
		 for the successful use of this technique in functional-type
		 a posteriori error estimates.
		 
		 \end{remark}

\subsection{Newton's algorithm} 
Newton's algorithm for solving the nonlinear variational 
problem (\ref{problem: discrete primal problem})
belonging to refinement level $l$ 
is given by Algorithm~\ref{newton_algorithm}.
Below we identify $\mathcal{A}(u_h^{l,k})$ with the corresponding vector
with respect to the chosen basis when we compute 
$\Vert \mathcal{A}(u_h^{l,k}) \Vert_{\ell_\infty}$.

\begin{algorithm}[H]
	\caption{Newton's algorithm on level $l$}\label{newton_algorithm}
	\begin{algorithmic}[1]
		\State 
		       Start with some initial guess $u^{l,0}_h \in
                         U_h^l$, set $k=0$,
                       and set $TOL_{Newton}^l > 0$.
		\While{$\Vert \mathcal{A}(u_h^{l,k}) \Vert_{\ell_\infty}> TOL_{Newton}^l$}
		\State Solve for $\delta u^{l,k}_h$, 
		       $$ \mathcal{A}'(u^{l,k}_h)(\delta u^{l,k}_h,v_h)=-\mathcal{A}(u^{l,k}_h)(v_h)  \quad \forall v_h \in 
		       V_h^l.$$
		\State Update : $u^{l,k+1}_h=u^{l,k}_h+\alpha \delta u^{l,k}_h$ for some good choice $\alpha \in (0,1]$.
		\State $k=k+1.$
		\EndWhile
	\end{algorithmic}
\end{algorithm}
	\begin{remark}
		In order to save computational cost we do not rebuild the matrices in every step.
		We rebuild the matrices if $\Vert \mathcal{A}(u_h^{l,k}) \Vert_{\ell_\infty}/\Vert \mathcal{A}(u_h^{l,k-1}) \Vert_{\ell_\infty} > 0.85 $ in Algorithm \ref{newton_algorithm}.
	\end{remark}

\begin{remark}
	Motivated by nested iterations, 
	see, e.g., Section~6 in \cite{BeRa01},
	and the analysis for nonlinear nested iterations 
	as given in Section~9.5 from \cite{MultigridHackusch03a},
	we use  $TOL_{Newton}^1= 10^{-8}\Vert \mathcal{A}(u_h^{1,0}) \Vert_{\ell_\infty}$ and $TOL_{Newton}^l= 10^{-2}\Vert \mathcal{A}(u_h^{l,0}) \Vert_{\ell_\infty}$ for $l > 1$  as stopping criteria.
\end{remark}

\begin{remark}
The parameter $\alpha$ can be obtained 
by means of
a line search procedure. To obtain a good convergence,
 we used  $\alpha=\gamma^L$ with $0<\gamma<1$, where the smallest $L$
 that
fulfills 
 $$\Vert \mathcal{A}(u_h^{l,k}+\alpha \delta u^{l,k}_h ) \Vert_{\ell_\infty}<\text{c}(L,L_{max})\Vert \mathcal{A}(u_h^{l,k} ) \Vert_{\ell_\infty},$$ 
with 
$$c(L,L_{max}):=\begin{cases}
0.8\quad &L=0\\
0.888 \quad &L=1\\
(0.888+0.112\sqrt{\frac{L+1}{L_{max}}}) \quad &L>1
\end{cases},$$
 $L=\{0, 1, 2 \cdots, L_{max}-1\}$ and $L_{max}=200$, is accepted.  
This choice of $\alpha$ was taken heuristically to obtain a
better convergence of the Newton method in the numerical Example
\ref{subsubsection Example 1c}. 
In Algorithm \ref{newton_algorithm}, we choose
$\gamma=0.9$, and in Algorithm \ref{inexat_newton_algorithm_for_multiple_goal_functionals},
 $\gamma=0.85$.
We remark that a standard backtracking line search method also
  works, see, e.g., \cite{ToWi2017}, but the previous exotic choice yields better iteration numbers.
\end{remark}


\subsection{Adaptive Newton algorithms for multiple-goal functionals }

In this section, we describe the key algorithm. The basic structure 
of the algorithm is similar to 
that
presented in \cite{RanVi2013} and \cite{ErnVohral2013}.
Our contribution is the extension to multiple-goal functionals.

	\begin{algorithm}[H]
		\caption{Adaptive Newton algorithm for multiple-goal
                  functionals on level $l$ } \label{inexat_newton_algorithm_for_multiple_goal_functionals}
		\begin{algorithmic}[1]
			\State Start with  some initial guess $u^{l,0}_h \in U_h^l$ and $k=0$.
				\State For $z^{l,0}_h$, solve $$ \mathcal{A}'(u^{l,0}_h)(v_h,z^{l,0}_h)=(J_{\mathfrak{E}}^{(0)})'(u^{l,0}_h)(v_h) \quad \forall v_h \in V_h^l,$$
				with  $(J_{\mathfrak{E}}^{(0)})'$ constructed with $u^{l,(2)}_h$ and $u^{l,0}_h$ as defined in (\ref{ErrorrepresentationFunctionalapprox'}).
			\While{$|\mathcal{A}(u^{l,k}_h)(z^{l,k}_h)|> 10^{-2} \eta_h^{l-1}$}
		\State For $\delta u^{l,k}_h$, solve $$ \mathcal{A}'(u^{l,k}_h)(\delta u^{l,k}_h,v_h)=-\mathcal{A}(u^{l,k}_h)(v_h)  \quad \forall v_h \in V_h^l.$$
		\State Update : $u^{l,k+1}_h=u^{l,k}_h+\alpha \delta u^{l,k}_h$ for some good choice $\alpha \in (0,1]$.
		\State $k=k+1.$
		\State For $z^{l,k}_h$, solve $$ \mathcal{A}'(u^{l,k}_h)(v_h,z^{l,k}_h)=(J_{\mathfrak{E}}^{(k)})'(u^{l,k}_h)(v_h) \quad \forall v_h \in U_h^l,$$
		$\quad$ with $(J_{\mathfrak{E}}^{(k)})'$ constructed with $u^{l,(2)}_h$ and $u^{l,k}_h$ as in (\ref{ErrorrepresentationFunctionalapprox'}).
			\EndWhile
		\end{algorithmic}
	\end{algorithm}



\subsection{The final algorithm} 
Now let us compose the final adaptive algorithm that starts from an initial mesh
$\mathcal{T}_h^1$ and the corresponding finite element spaces 
$V_h^1$,  $U_h^1$, $U_{h}^{1,(2)}$ and $V_{h}^{1,(2)}$, where $U_{h}^{1,(2)}$ and $V_{h}^{1,(2)}$ are the
	enriched finite element spaces as described 
	in Section~\ref{subsec_SpatialDiscretization}.
The refinement procedure produces a sequence of finer and finer 
meshes
$\mathcal{T}_h^l$ with the correponding FE spaces
$V_h^l$,  $U_h^l$, $U_{h}^{l,(2)}$ and $V_{h}^{l,(2)}$ for $l=2,3,\ldots$ .

	\begin{algorithm}[H]
		\caption{The final algorithm }\label{final algorithm}
		\begin{algorithmic}[1]
				\State Start with some initial guess 
                                  $u_{h}^{0,(2)}$,$u_h^{0}$, set $l=1$
                                  and set $TOL_{dis} > 0$.
				\State Solve  (\ref{problem: discrete primal problem}) for  $u_h^{l,(2)}$ using Algorithm \ref{newton_algorithm} with the initial guess $u_h^{l-1,(2)}$ on the discrete space $U_{h}^{l,(2)}$. \label{solve Uh2}
				\State Solve (\ref{problem: discrete primal problem}) and (\ref{Problem: discrete adjoint on small space}) using Algorithm \ref{inexat_newton_algorithm_for_multiple_goal_functionals} with the initial guess $u_h^{l-1}$ on the discrete spaces $U_{h}^l$ and $V_{h}^l$ . \label{final algorithm: solveprimal}
				\State Construct the combined functional $J_{\mathfrak{E}}$ as in (\ref{ErrorrepresentationFunctionalapprox}).
				\State Solve the adjoint problem (\ref{Problem: adjoint}) for $J_{\mathfrak{E}}$ on $V_h^{l,(2)}$. \label{final algorithm: solveadjoint}
				\State Construct the error estimator $\eta_K$  by distributing $\eta_i$ defined in (\ref{eta_i_PU}) to the elements.
				\State Mark elements with some refinement strategy. \label{final algorithm: refinement}
				\State Refine marked elements: $\mathcal{T}_h^l \mapsto\mathcal{T }_h^{l+1}$ and $l=l+1$.
				\State If $|\eta_h| < TOL_{dis}$ stop, else go to \ref{solve Uh2}.
		\end{algorithmic}
	\end{algorithm}

In step~\ref{final algorithm: solveprimal} of Algorithm~\ref{final algorithm}, 
we replaced the estimated error $\eta_h^l$ by  $\eta_h^{l-1}$
in  Algorithm~ \ref{inexat_newton_algorithm_for_multiple_goal_functionals},
because we want to avoid the solution of the adjoint problem on the space $V_h^{l,(2)}$.
	Since the error in the previous estimate might be larger in general, 
        we take $10^{-2} \eta_h^{l-1}$ instead of $10^{-1} \eta_h^{l}$, 
        which was suggested in \cite{RanVi2013}.

	Thus,  $\eta_h^{l-1}$ is not defined on the first 
        level.
        Therefore, we set it to $\eta_h^{0}:=10^{-8}$. This means 
	that we perform
	more iterations on the coarsest level. However, solving on this level is very cheap.

\begin{remark}
	We notice that step \ref{solve Uh2} in Algorithm \ref{final algorithm} is costly, because 
	we have to solve a problem corresponding to an enriched finite element space.
\end{remark}

\begin{remark}
In step \ref{final algorithm: refinement} of Algorithm \ref{final algorithm}, we mark all elements $K'$ where $\eta_{K'} \leq \frac{1}{|\mathcal{T}_h^l|}\sum_{K \in \mathcal{T}_h^l}^{}\eta_K$, where $|\mathcal{T}_h^l|$ denotes the number of elements.
\end{remark}

\begin{remark}
Inspecting Algorithm \ref{final algorithm}, we need solve 
at each refinement level 
four problems: two are solved in  step \ref{final algorithm: solveprimal}, and
one in step \ref{solve Uh2} and \ref{final algorithm: solveadjoint},
respectively.
  On the one hand,
this is costly in comparison to other error estimators, e.g., residual-based, 
where only the primal problem needs to be solved. On the other hand, the 
adjoint solutions yield precise sensitivity measures for accurate measurements 
of the goal functionals. In addition, we control both the discretization
and nonlinear iteration error for multiple goal functionals.
Finally, the proposed approach is nonetheless much cheaper for many goal functionals. 
A naive approach 
(for a discussion in the linear case of multiple goal functionals 
or for
using the primal part of the error estimator only, we refer 
the reader again to \cite{HaHou03,Ha08})
would mean to solve $2N+2$ problems (i.e., $N+1$ for the primal
part). 
\end{remark}

%
\section{Numerical examples}
\label{sec_num_tests}
%
In this section, we perform numerical tests for two nonlinear problems,
where the first problem contains 
two model parameters.
We consider different choices
of these parameters that lead to different levels of difficulty with respect 
to their numerical treatment.

\begin{itemize}
	\item \textbf{Example 1} ($p$-Laplacian):
	\begin{enumerate}[a)]
		\item Smooth solution with homogeneous Dirichlet boundary conditions and right hand side on the unit square  for $p=2$ and $p=4$ with $\varepsilon=1$ as regularization parameter, 
		and an integral evaluation over the whole domain as functional of interest.
		\item Smooth solution with inhomogeneous Dirichlet boundary conditions on the unit square with a disturbed grid and $p=5$ and $p=1.5$ with 
		$\varepsilon=0.5$ 
		and a point evaluation as functional of interest.
		\item Solution with corner singularities and homogeneous Dirichlet boundary conditions on a cheese domain with $p=4$ and $p=1.33$ with a very small regularization parameter $\varepsilon=10^{-10}$, 
		and two nonlinear and two linear functionals of interest.
	\end{enumerate}
	\item \textbf{Example 2} (a quasilinear PDE system): \newline
	 Solution with low regularity on a slit domain with mixed boundary conditions,
	 and one linear and five nonlinear functionals of interest.
\end{itemize}

The implementation is based on the finite element library 
deal.II \cite{dealII84} and the extension of our previous work \cite{EnWi17}.

\subsection{Preliminaries}

The following examples are discretized using globally continuous  isoparametric 
quadrilateral
elements as 
introduced in Section~\ref{subsec_SpatialDiscretization}.
	If not mentioned otherwise, we use 
	$U_h^{(2)}=Q^{r+1}_c \cap U$ and $V_h^{(2)}=Q^{r+1}_c \cap V$
	for the enriched finite element spaces, 
	if $U_h=Q^r_c \cap U$ and $V_h=Q^r_c \cap V$
	is used for the original finite element spaces.
In all numerical experiments we used $r=1$ except in Section \ref{subsubsection: Example1a} Case 1, where the used discretization is given explicitly.
To solve the arising linear systems, we used the sparse direct solver UMFPACK \cite{UMFPACK}. 
The error-weighting function
        $\mathfrak{E}(x,\vec{J}(u_h)):=\sum_{i=1}^{N}\frac{x_i}{|J_i(u_h)|}$ is used to construct $J_\mathfrak{E}$ as in (\ref{ErrorrepresentationFunctionalapprox}). 
In our computations,  we used the finite element function which is $1$ at the nodes which do not
belong to the Dirichlet boundary and fulfills the boundary conditions at the
nodes which belongs to the Dirichlet boundary  as initial guess for
$u_{h}^{0,(2)}$ and $u_h^{0}$. 

To investigate how well our error estimator performs in estimating 
the error, we introduce the effectivity indices for the functional $J$
as follows:
\begin{align}
	I_{eff} &:= \frac{\eta_h}{|J(u)-J(u_h)|},\label{Ieff}\\ 
	I_{effp} &:= \frac{|\rho(\tilde{u})(z_h^{(2)}-z_h)|}{|J(u)-J(u_h)|},\label{Ieffp}\\ 
	I_{effa} &:= \frac{|\rho^*(\tilde{u},\tilde{z})(u_h^{(2)}-u_h)|}{|J(u)-J(u_h)|}, \label{Ieffa}
\end{align}
where $\rho$ is defined by (\ref{Error Estimator: primal}), 
$\rho^*$ as in (\ref{Error Estimator: adjoint}), 
and $\eta_h$ as in (\ref{Error Estimator: discretization inexact}).
We call (\ref{Ieff}) the effectivity index, 
(\ref{Ieffp}) the primal effectivity index, 
and (\ref{Ieffa}) the adjoint effectivity index.
In the first part, we analyze the behavior of our algorithm for the regularized $p$-Laplace equation (\ref{p_laplace_problem}).
In Section~\ref{subsubsection: Example1a}, Case 1, 
we apply our algorithm to the linear problem given in \cite{RiWi15_dwr}, i.e., for $p=2$. 
For Section~\ref{subsubsection: Example1a},  Case 2, we chose $p=4, \varepsilon=1$,
and apply our algorithm to a nonlinear problem,
and compare the refinement evolution for the different error estimators  
$|\rho(\tilde{u})(z_h^{(2)}-z_h)|,$ $|\rho^*(\tilde{u},\tilde{z})(u_h^{(2)}-u_h)|$ and $\eta_h$.

In Section \ref{subsubsection: Example 1b}, we solve the $p$-Laplace equation for $p=5$ and $p=1.5$ on a disturbed grid, 
aiming for a point evaluation.  
We compare the results of our algorithm 
with the results of global refinement and also to the different error estimators. 
The examples in Section \ref{subsubsection Example 1c} consider several reentrant corners, several nonlinear functionals, and a very small regularization parameter $\varepsilon =10^{-10 }$.
In Section~\ref{QuasilinearvectorPDE}, 
we investigate the behavior of our algorithm for a quasilinear PDE system.

\subsection{Example 1: $p$-Laplace}
\label{subsection: Example1}
Let $\varepsilon >0$ and $p \in \mathbb{R}$ with $p > 1$, and let $\Omega$ be a bounded Lipschitz domain in $\mathbb{R}^2$. 
We again consider the Dirichlet problem for $p$-Laplace equation, cf. Section~\ref{Discretization}, 
but now with inhomogeneous Dirichlet boundary conditions: Find $u$ such that:
\begin{align}\label{p_laplace_problem}
	\begin{aligned}
	-\text{div}((\varepsilon^2 + |{\nabla u}|^2)^{\frac{p-2}{2}}\nabla u) &=f \quad  \forall  \text{in }\Omega, \\
	u &= g \quad \text{on } \partial \Omega .
	\end{aligned}
\end{align}

The Fr\'echet derivative $\mathcal{A}'(u)$ at $u$ of the nonlinear operator $\mathcal{A}$ 
corresponding to the $p$-Laplace problem problem~\ref{p_laplace_problem}, cf. also Section~\ref{Discretization},
is given by the variational identity

%
%
\begin{align*}
			\mathcal{A}'(u)(q,v) =& \langle{(\varepsilon^2+\Vert{\nabla u}\Vert_{\ell_2}^2)^{\frac{p-2}{2}}\nabla q,\nabla v}\rangle\\
			+& \langle{(p-2)(\varepsilon^2+\Vert{\nabla u}\Vert_{\ell_2}^2)^{\frac{p-4}{2}}(\nabla u, \nabla q)_{\ell_2} \nabla u,\nabla v}\rangle \quad  \forall q,v \in W^{1,p}_0(\Omega).
		\end{align*}
%
		


\subsubsection{Regular cases}
\label{subsubsection: Example1a}
Here we consider a problem with a smooth solution and a smooth adjoint solution.
\newline

\textbf{Case 1 ($p=2$, i.e. Poisson problem):}
This is the same example as Example 1 in \cite{RiWi15_dwr}. 
In this example,
the data
are given by $\Omega = (0,1) \times (0,1)$, $f=1$ and $g=0$.
We are interested in the following functional evaluation:
\begin{align*}
	J_1(u)&:=\int_{\Omega}u(x)\,d x \approx 0.03514425375\pm 10^{-10} .
\end{align*}
This reference value was taken from \cite{RiWi15_dwr}. 
If we compare our results in Table \ref{table: p=2 RiWiQ3Q6} with the results in \cite{RiWi15_dwr},
 then we observe that they are quite similar.
  The estimated error $\eta_h$  is almost the same, and the DOFs 
  exactly coincide with the 
 DOFs
  in \cite{RiWi15_dwr}. 
  However, using
  just one polynomial degree higher for $U_h^{(2)}$, 
  we obtain similar results with less computational cost 
  as is shown in Table \ref{table: p=2 RiWiQ3Q4}.

\begin{figure}[H]
	\centering
\begin{tabular}{|c|r|c|c|l|l|l|}
	\hline
	$l$ & \text{DOFs} & $|J(u)-J(u_h)|$ & $\eta_h$ & $I_{eff}$ & $I_{effp}$ & $I_{effa}$ \\ \hline
	1   & 169    & 8.51E-07        & 8.47E-07 & 1.00         & 1.00       & 1.00          \\ \hline
	2   & 317    & 1.12E-07        & 1.37E-07 & 1.23      & 1.23       & 1.23       \\ \hline
	3   & 937    & 5.57E-09        & 7.55E-09 & 1.35      & 1.35       & 1.36       \\ \hline
	4   & 1 813   & 1.15E-09        & 1.41E-09 & 1.22      & 1.23       & 1.22       \\ \hline
	5   & 3 877   & 6.48E-11        & 8.05E-11 & 1.24      & 1.24       & 1.24       \\ \hline
	6   & 7 057   & 2.81E-11        & 2.07E-11 & 0.74      & 0.74       & 0.74       \\ \hline
\end{tabular}
	\captionof{table}{Section \ref{subsubsection: Example1a}, Case
          1. Display of exact error $|J(u)-J(u_h)|$ , estimated error $\eta_h$, and effectivity indices   for  $U_h=Q^3_c$ and $U_h^{(2)}=Q^6_c$.}
	\label{table: p=2 RiWiQ3Q6}
\end{figure} 
\begin{figure}[H]
	\centering
\begin{tabular}{|c|r|c|c|l|l|l|}
	\hline
	$l$ & \text{DOFs} & $|J(u)-J(u_h)|$ & $\eta_h$ & $I_{eff}$ & $I_{effp}$ & $I_{effa}$ \\ \hline
		1   & 169    & 8.51E-07    & 7.72E-07   & 0.91      & 0.91       & 0.91       \\ \hline
		2   & 317    & 1.12E-07    & 1.32E-07   & 1.18      & 1.18       & 1.18       \\ \hline
		3   & 789    & 5.12E-08    & 5.33E-08   & 1.04      & 1.04       & 1.04       \\ \hline
		4   & 1 301   & 4.11E-09    & 4.06E-09   & 0.99      & 0.99       & 0.99       \\ \hline
		5   & 1 977   & 1.06E-09    & 1.58E-09   & 1.49      & 1.49       & 1.5        \\ \hline
		6   & 4 149   & 6.56E-11    & 7.91E-11   & 1.2       & 1.2        & 1.21       \\ \hline
		7   & 7 273   & 2.65E-11    & 2.11E-11   & 0.8       & 0.8        & 0.8        \\ \hline
	\end{tabular}
	\captionof{table}{Section \ref{subsubsection: Example1a}, Case
          1. Display of exact error $|J(u)-J(u_h)|$ , estimated error $\eta_h$, and effectivity indices   for  $U_h=Q^3_c$ and $U_h^{(2)}=Q^4_c$.}
	\label{table: p=2 RiWiQ3Q4}
\end{figure}

\noindent
\textbf{Case 2 ($p=4,$ $\varepsilon =1$)}:
We use the same 
	setting
as above, but with $p=4$ and $\varepsilon=1$.
The finite element spaces are given by $U_h=Q^1_c$ and  $U_h^{(2)}=Q^2_c$ .
We are interested in the following functional evaluation
\begin{align*}
J_1(u)&:=\int_{\Omega}u(x) \,d x \approx  0.033553988572 \pm 10^{-6}.
\end{align*}
This reference value was computed on a fine grid with $263\,169$ $ \text{DOFs}$ ($9$ global refinement steps).
In this example, we compare the refinements for different error estimators.

\begin{figure}[H]
	\centering
		\begin{tabular}{|c|r|c|l|l|l|}
			\hline
			$l$ & \text{DOFs} & $|J(u)-J(u_h)|$ & $I_{eff}$ & $I_{effp}$ & $I_{effa}$ \\ \hline
		1  & 9      & 1.08E-02    & 0.98      & 0.92       & 1.05       \\ \hline
		2  & 25     & 2.82E-03    & 0.99      & 0.92       & 1.07       \\ \hline
		3  & 81     & 7.11E-04    & 1.00         & 0.92       & 1.08       \\ \hline
		4  & 289    & 1.78E-04    & 1.00       & 0.92       & 1.08       \\ \hline
		5  & 1 089   & 4.44E-05    & 1.00        & 0.92       & 1.09       \\ \hline
		6  & 4 193   & 1.15E-05    & 1.07      & 0.98       & 1.15       \\ \hline
		7  & 6 545   & 9.45E-06    & 1.08      & 0.99       & 1.17       \\ \hline
		8  & 16 769  & 2.61E-06    & 1.07      & 0.98       & 1.16       \\ \hline
		9  & 36 009  & 1.75E-06    & 1.13      & 1.04       & 1.22       \\ \hline
	\end{tabular}
	\captionof{table}{Section \ref{subsubsection: Example1a}, Case 2. Refinement 
	is only based on
	the primal part of the error estimator $\eta_h$.}\label{tableP4Primal}
\end{figure}

\begin{figure}[H]
	\centering
	\begin{tabular}{|c|r|c|l|l|l|}
		\hline
		$l$ & \text{DOFs} & $|J(u)-J(u_h)|$ & $I_{eff}$ & $I_{effp}$ & $I_{effa}$ \\ \hline
1   & 9      & 1.08E-02    & 0.98      & 0.92       & 1.05       \\ \hline
2   & 25     & 2.82E-03    & 0.99      & 0.92       & 1.07       \\ \hline
3   & 81     & 7.11E-04    & 1.00         & 0.92       & 1.08       \\ \hline
4   & 289    & 1.78E-04    & 1.00         & 0.92       & 1.08       \\ \hline
5   & 913    & 7.54E-05    & 1.15      & 1.09       & 1.21       \\ \hline
6   & 1 545   & 4.08E-05    & 1.09      & 1 .00         & 1.18       \\ \hline
7   & 4 225   & 1.10E-05    & 1.02      & 0.93       & 1.10      \\ \hline
8   & 10 513  & 6.56E-06    & 1.10       & 1.04       & 1.16       \\ \hline
9   & 20 649  & 2.48E-06    & 1.12      & 1.03       & 1.22       \\ \hline
\end{tabular}
	\captionof{table}{Section \ref{subsubsection: Example1a}, Case 2. Refinement 
	is only based on
	the adjoint part of the error estimator $\eta_h$.}\label{tableP4Adjoint}
\end{figure}

\begin{figure}[H]
	\centering
	\begin{tabular}{|c|r|c|l|l|l|}
		\hline
		$l$ & \text{DOFs} & $|J(u)-J(u_h)|$ & $I_{eff}$ & $I_{effp}$ & $I_{effa}$ \\ \hline
		1   & 9      & 1.08E-02    & 0.98      & 0.92       & 1.05       \\ \hline
		2   & 25     & 2.82E-03    & 0.99      & 0.92       & 1.07       \\ \hline
		3   & 81     & 7.11E-04    & 1.00      & 0.92       & 1.08       \\ \hline
		4   & 289    & 1.78E-04    & 1.00      & 0.92       & 1.08       \\ \hline
		5   & 1 089   & 4.44E-05    & 1.00      & 0.92       & 1.09       \\ \hline
		6   & 3 137   & 2.26E-05    & 1.14      & 1.07       & 1.20       \\ \hline
		7   & 5 833   & 1.02E-05    & 1.10      & 1.01       & 1.19       \\ \hline
		8   & 16 641  & 2.61E-06    & 1.07      & 0.98       & 1.15       \\ \hline
		9   & 38 993  & 1.59E-06    & 1.15      & 1.08       & 1.21       \\ \hline
	\end{tabular}
	\captionof{table}{Section \ref{subsubsection: Example1a}, Case 2. 
          Refinement for the error estimator $\eta_h$.}
        \label{tableP4PrimalAdjoint}
\end{figure}

\newpage
In this example, 
we obtain quite good effectivity indices for the refinements 
based on the 
the primal part of the error estimator, cf. Table \ref{tableP4Primal}, 
the adjoint part of the error estimator, cf. Table \ref{tableP4Adjoint}, 
and the full error estimator $\eta_h$,  cf. Table \ref{tableP4PrimalAdjoint}. 
Furthermore, the convergence rates are also very similar. 
One might conclude that the adjoint error estimator is not required to obtain good effectivity indices. 
However, in the following examples, we observe that this is not the case
for less regular solutions and adjoint solutions.


\subsubsection{Semiregular cases}
\label{subsubsection: Example 1b}
As in the regular cases, we consider a smooth solution, but a low regular adjoint solution.
This example is motivated by an example in \cite{ToWi2017}. We choose the right-hand side and 
the boundary conditions such that exact solution is given by $ u(x,y)=\text{sin}(6x+6y).$
The computation was done on the unit square $\Omega = (0,1) \times (0,1)$ on a slightly
perturbed mesh (generated with the deal.II \cite{BangerthHartmannKanschat2007,dealII84} command 
\verb|distort_random| with $0.2$ on a $4$ times globally refined grid unit square). The resulting mesh is shown in Figure \ref{Example 1c: initial mesh}. The functional of interest is $J(u)=u(0.6,0.6)$.
We consider the following two cases:
\begin{itemize}
	\item \textbf{Case 1 ($p=5,\varepsilon = 0.5$)},
	\item \textbf{Case 2 ($p=1.5,\varepsilon =0.5$)}.
\end{itemize}
In both cases, the method also worked for the perturbed meshes. 
For the case $p=5$ and $\varepsilon=0.5$, we 
observe from Figure~\ref{Example1b: adjoint solution} that the   adjoint solution almost vanishes in the set outside the domain which is covered by the condition 
$\nabla u =0$, and contains the point $(0.6,0.6)$. 
This was not observed in Case\,2. 
However, the condition $\nabla u=0$ seems to be important in both cases. 
The adaptively refined meshes shown in Figure \ref{Example1b: primal solution} and Figure \ref{Example1b: 2:refined+marked} have more refinement levels in these regions. 
In Figure \ref{Example1bCase1gnuplot} and Figure\,\ref{Example1bCase2gnuplot},
we observe 
that we get the same convergence rate 
as in the case of
uniform refinement. Since the solution is smooth, a global refinement already attains the optimal convergence rate. 
However, we 
get a reduction of the number of DOFs that are needed to obtain the same error.
%
Furthermore, 
we monitor that the effectivity index is better on finer meshes.
The reason might be the neglected remainder term from Theorem \ref{Theorem: Error Representation}. 
From Table \ref{Example1bCase1Ieff} and Table \ref{Example1bCase2Ieff}, we conclude that 
this does not necessarily  hold
for the primal and the adjoint error estimator separately. 
\begin{figure}[H]
		\begin{minipage}[t]{0.45\linewidth}
\hspace{-0.15 \linewidth}
			\includegraphics[scale = 0.30]{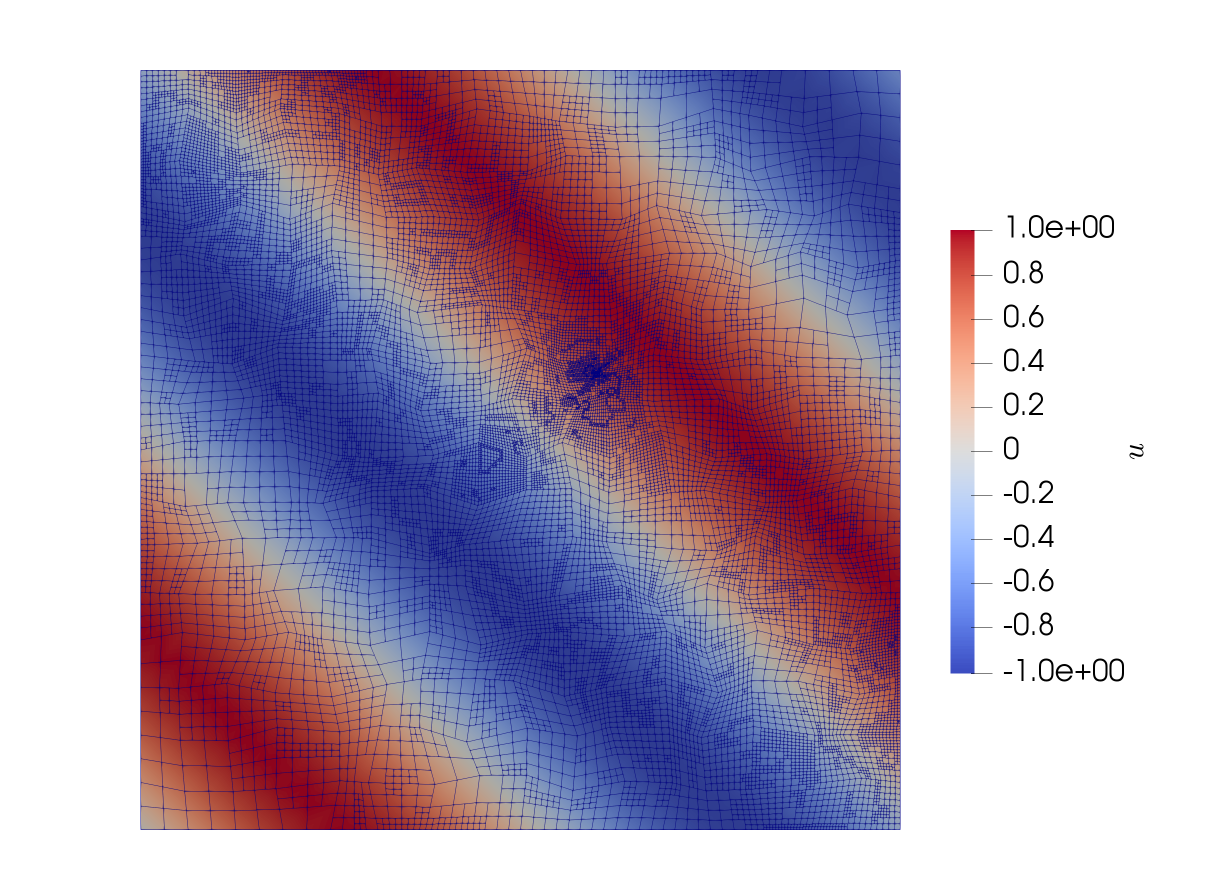}
			\caption{Section \ref{subsubsection: Example 1b}, Case 1. Primal solution and mesh after six adaptive refinements.} 
			\label{Example1b: primal solution}
		\end{minipage}
							\hspace{0.1 \linewidth}
				\begin{minipage}[t]{0.45\linewidth}
\hspace{-0.15 \linewidth}
					\includegraphics[scale = 0.30]{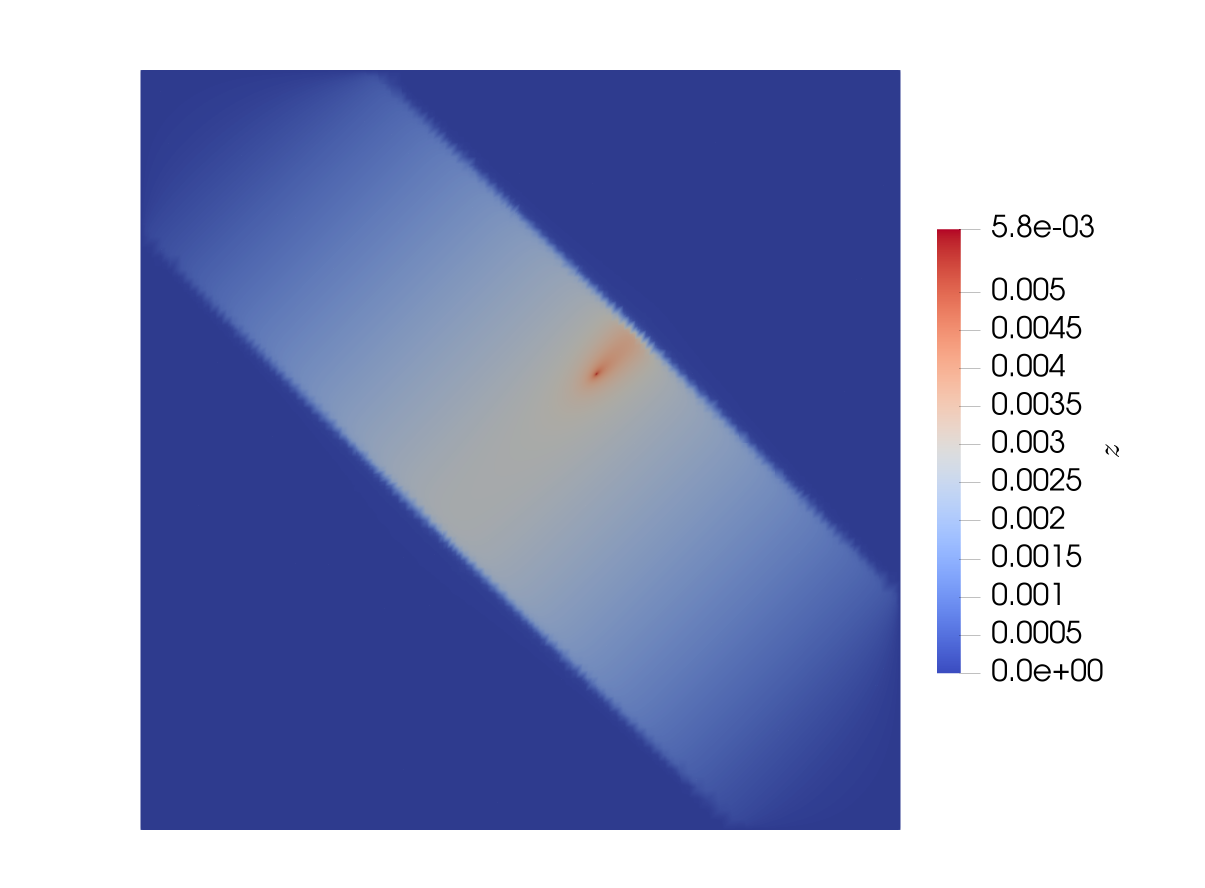}
					\caption{Section \ref{subsubsection: Example 1b}, Case 1. Adjoint solution on the mesh as given in Figure \ref{Example1b: primal solution}.}\label{Example1b: adjoint solution}
				\end{minipage}

\end{figure}
\begin{figure}[H]
	
	\begin{minipage}[t]{0.45\linewidth}
		\centering
		\scalebox{0.65}{\input{Example1bCase1gnuplottex}}
		\caption{Section \ref{subsubsection: Example 1b}, Case 1. Error vs \text{DOFs} for $p=5$ and $\varepsilon = 0.5$.} \label{Example1bCase1gnuplot}
	\end{minipage}
	\hspace{0.1 \linewidth}
	\begin{minipage}[t]{0.45\linewidth}
				\scalebox{0.65}{\input{Example1bCase2gnuplottex}}
			\caption{Section \ref{subsubsection: Example 1b}, Case 2. Error vs DOFs for $p=1.5$ and $\varepsilon = 0.5$.}\label{Example1bCase2gnuplot}
	\end{minipage}
%
	
\end{figure}
\begin{figure}[H]

		\begin{minipage}[t]{0.45\linewidth}
				\tiny				
					\begin{tabular}{|r|r|c|l|l|l|}
						\hline
						$l$ & \text{DOFs}  & $|J(u)-J(u_h)|$& $I_{eff}$ & $I_{effp}$ & $I_{effa}$ \\ \hline
						1  & 289     & 4.24E-03    & 0.48      & 0.60       & 1.56       \\ \hline
						2  & 599     & 5.23E-03    & 0.80      & 0.63       & 0.98       \\ \hline
						3  & 1 095    & 7.72E-04    & 0.16      & 0.02       & 0.34       \\ \hline
						4  & 2 418    & 8.52E-05    & 1.81      & 2.58       & 1.03       \\ \hline
						5  & 4 918    & 1.28E-05    & 4.92      & 3.35       & 6.49       \\ \hline
						6  & 10 112   & 2.64E-05    & 2.03      & 1.83       & 2.22       \\ \hline
						7  & 20 068   & 3.46E-06    & 5.59      & 10.33      & 0.86       \\ \hline
						8  & 40 302   & 1.02E-05    & 1.66      & 2.16       & 1.16       \\ \hline
						9  & 79 468   & 4.45E-06    & 1.51      & 1.60       & 1.43       \\ \hline
						10  & 157 272  & 2.68E-06    & 1.62      & 1.68       & 1.55       \\ \hline
						11 & 305 901  & 1.36E-06    & 1.32      & 1.62       & 1.01       \\ \hline
						12 & 602 720  & 8.52E-07    & 1.29      & 1.46       & 1.12       \\ \hline
						13 & 1 157 353 & 3.40E-07    & 1.28      & 1.55       & 1.01       \\ \hline
					\end{tabular}
	\captionof{table}{Section \ref{subsubsection: Example 1b}, Case 1. Effectivity indices  for $p=5$ and $\varepsilon=0.5$.}\label{Example1bCase1Ieff}
		\end{minipage}
		\hspace{0.1\linewidth}
				\begin{minipage}[t]{0.45\linewidth}
					\tiny				
						\begin{tabular}{|r|r|c|l|l|l|}
							\hline
							$l$ & \text{DOFs} & $|J(u)-J(u_h)|$ & $I_{eff}$ & $I_{effp}$ & $I_{effa}$ \\ \hline
							1  & 289    & 2.07E-02    & 0.61      & 0.47       & 0.75       \\ \hline
							2  & 503    & 5.55E-03    & 0.72      & 0.89       & 0.55       \\ \hline
							3  & 994    & 2.88E-03    & 0.89      & 1.30       & 0.49       \\ \hline
							4  & 2 090   & 8.55E-04    & 1.23      & 1.50       & 0.96       \\ \hline
							5  & 4 233   & 4.34E-04    & 1.45      & 1.90       & 1.00       \\ \hline
							6  & 8 667   & 1.42E-04    & 1.34      & 1.88       & 0.80       \\ \hline
							7  & 17 276  & 8.14E-05    & 1.40      & 2.71       & 0.09       \\ \hline
							8  & 34 846  & 3.54E-05    & 1.28      & 1.75       & 0.80       \\ \hline
							9  & 68 765  & 1.64E-05    & 1.36      & 2.58       & 0.14       \\ \hline
							10  & 136 267 & 8.59E-06    & 1.29      & 2.07       & 0.51       \\ \hline
							11 & 263 508 & 4.30E-06    & 1.19      & 2.20       & 0.18       \\ \hline
							12 & 514 223 & 2.18E-06    & 1.22      & 1.99       & 0.44       \\ \hline
							13 & 988 042 & 1.01E-06    & 1.22      & 2.20       & 0.24       \\ \hline
						\end{tabular}
					\captionof{table}{Section
                                          \ref{subsubsection: Example 1b}, 
                                          Case 2. Effectivity indices  for $p=1.5$
					and $\varepsilon=0.5$.} \label{Example1bCase2Ieff}
				\end{minipage}
\end{figure}
\begin{figure}[H]
	
	\begin{minipage}[t]{0.45\linewidth}
	\hspace{-0.15 \linewidth}			
		\includegraphics[scale = 0.22]{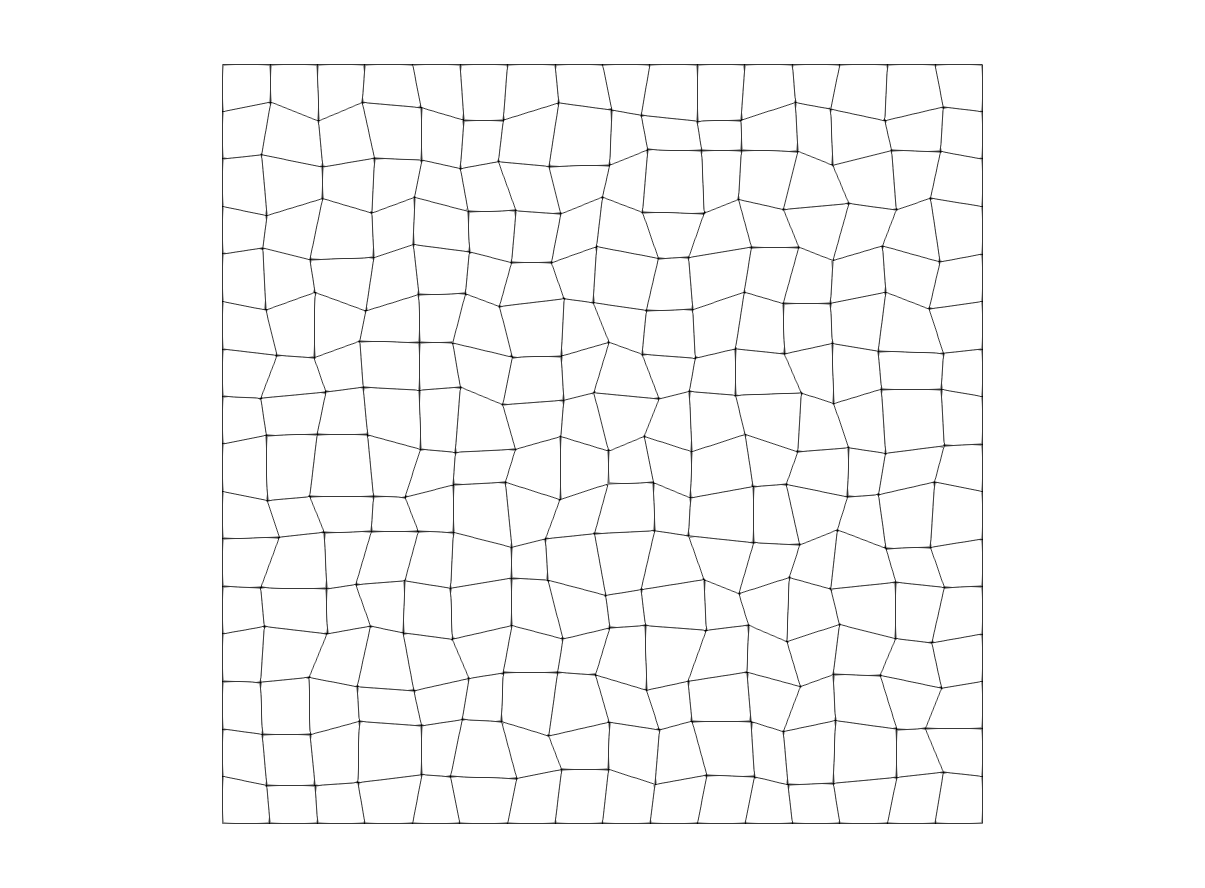}
		\caption{Disturbed initial mesh for Case 1 and Case 2 of Section \ref{subsubsection: Example 1b}.}\label{Example 1c: initial mesh}
	\end{minipage}
	\hspace{0.1\linewidth}
	\begin{minipage}[t]{0.45\linewidth}
		\hspace{-0.15 \linewidth}
		\includegraphics[scale = 0.22]{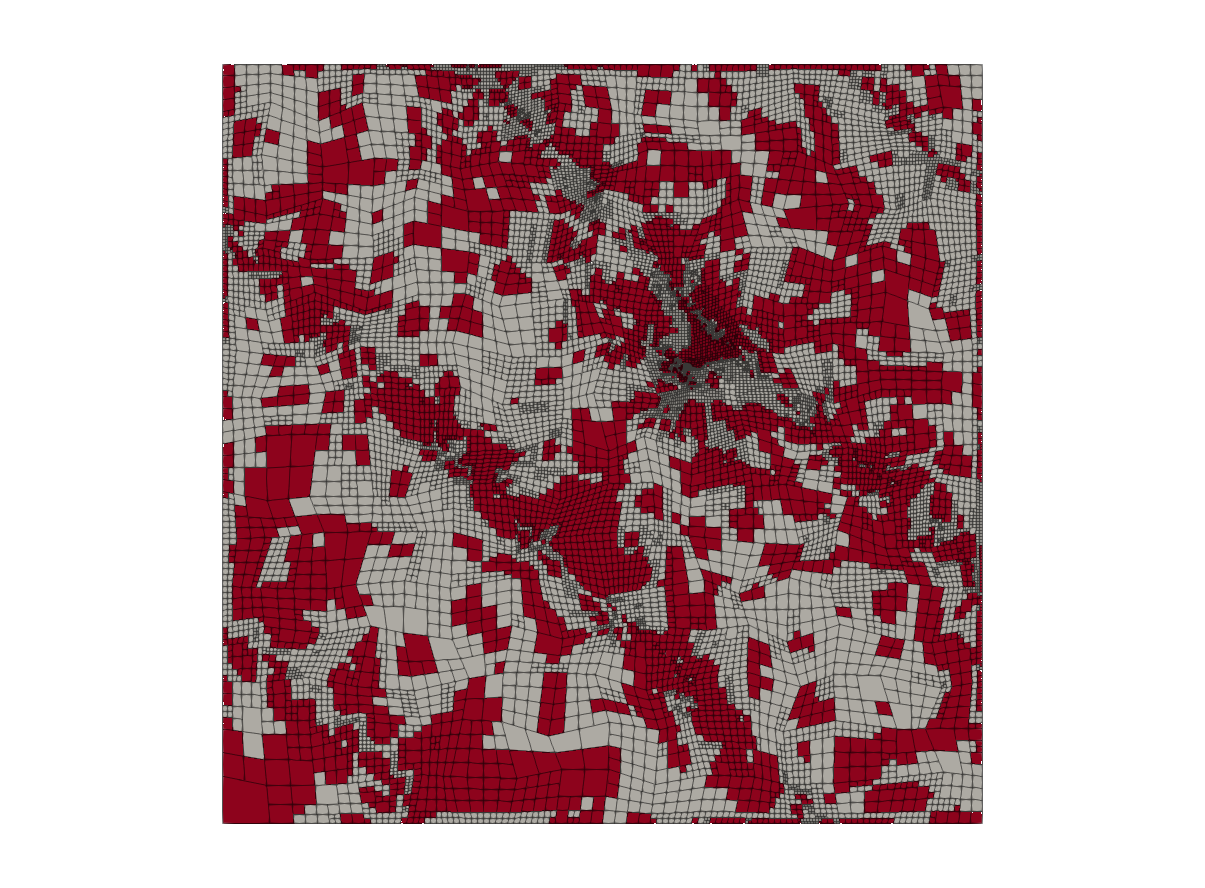}
		\caption{Marked elements (in red)  at refinement level $l= 7$ for Case 2 of Section \ref{subsubsection: Example 1b}.} \label{Example1b: 2:refined+marked}
	\end{minipage}
\end{figure}

 \subsubsection{Low regularity cases} 
\label{subsubsection Example 1c}
As in Section \ref{subsubsection: Example1a}, we consider homogeneous Dirichlet conditions and $f=1$ as right-hand side for the $p$-Laplace equation (\ref{p_laplace_problem}).
However, here both the solution and adjoint solution have low regularity.
 The initial mesh is given as in Figure \ref{initcheese}, which was
 constructed 
using the deal.II \cite{BangerthHartmannKanschat2007,dealII84} command cheese. 
With this data, we have singularities on each of the reentrant corners.  
Furthermore,  in this example, we chose the regularization parameter
$\varepsilon$  to be $10^{-10}$, which makes the problem very ill-conditioned
(in fact it is practically the original $p$-Laplace problem) where $\nabla u=0$, but it is very close to the unregularized $p$-Laplace problem as in \cite{Lind2006} and\cite{Glowinski1975}. 
We are interested in the following four goal functionals:
\begin{align*}
	J_1(u):= &(1+u(2.9,2.1))(1+u(2.1,2.9)),\\
	J_2(u):= &\left(\int_{\Omega} u(x,y)-u(2.5,2.5)\,d(x,y)\right)^2,\\	
	J_3(u):= &\int_{(2,3)\times(2,3)} u(x,y)\,d(x,y),\\
	J_4(u):= &u(0.6,0.6).
\end{align*}
These functionals will be combined 
to $J_{\mathfrak{E}}$ as formulated in (\ref{ErrorrepresentationFunctionalapprox}).\\
\noindent
\textbf{Case 1 ($p=4$, $\varepsilon=10^{-10}$):} 
First we consider a case where $p>2$. 
The following values, which were computed on a fine grid (8 global
refinements, $Q_c^2$ elements, 22\,038\,525 \text{DOFs}) on the cluster RADON1\footnote{\url{https://www.ricam.oeaw.ac.at/hpc/overview/}}, are used to compute the reference values:\\
\begin{minipage}[t]{0.5\linewidth}
	\begin{align*}
		\int_{\Omega} u(x,y)\, d(x,y)\approx & 4.1285036414 \pm 4\times10^{-5},\\
		\int_{(2,3)\times(2,3)} u(x,y)\,d(x,y)\approx & 0.31999986649\pm 10^{-5},\\
		u(2.9,2.1)\approx & 0.16071095234\pm 10^{-5},
	\end{align*}
\end{minipage}
\begin{minipage}[t]{0.5\linewidth}
	\begin{align*}
		u(2.9,2.1)\approx & 0.16071095234 \pm 10^{-5},\\
		u(0.6,0.6)\approx & 0.35554352679\pm 2\times 10^{-6},\\
		u(2.5,2.5)\approx & 0. 49244705234\pm 4\times10^{-6}.
	\end{align*}
\end{minipage}
\vspace*{0.5cm}
Considering the accuracy of the functional evaluations above, 
we observe that the relative errors in the functionals $J_1$, $J_2$, $J_3$ and  $J_4$ are less than $5\times  10^{-5}$.
Our algorithm yields
the results shown in Table \ref{Example 1b:
  errortable}. 
In Figure \ref{Example1c: Case 1 gnuplot2}, we can see that the absolute error in the error functional $J_\mathfrak{E}$ bounds the relative errors of the functionals $J_1$, $J_2$, $J_3$ and  $J_4$.
Furthermore, we  observe that $J_2$ is the dominating functional and $J_1$ is
the one with the smallest error on most refinement levels.
Therefore, we compare the convergence of this functionals in Figure
\ref{Example1c: Case 1 gnuplot}. 
For uniform refinement, we obtain an error behavior of approximately
$\mathcal{O}(\text{DOFs}^{-\frac{3}{4}})$ for $J_2$ 
and $\mathcal{O}(\text{DOFs}^{-\frac{3}{5}})$ for $J_1$, whereas we obtain
excellent 
convergence rates of 
$\mathcal{O}(\text{DOFs}^{-1})$ for both functionals using our refinement algorithm. 
We are not aware of a full convergence analysis on adaptive meshes for
  pointwise estimates for the $p$-Laplacian, but mention 
two related studies \cite{CaKlo03} for $p>2$ showing a posteriori 
estimates for the $W^{1,p}$ norm and \cite{Breit2017} with 
pointwise a priori estimates for the $p$-Laplacian.
The bad convergence of $J_2$ might result from the fact that the point
$(2.5,2.5)$ 
is the intersection of two lines, where the problem is ill-conditioned, and
also leads 
to a kink in the solution at this point (see Figure \ref{Example 1c: p=4
  solution}). 
This kink is not visible 
in the case $p=1.33$ (see Figure \ref{Example 1c: p=1.33
  solution}).  
Comparing the number of Newton steps in Table \ref{Example1c_inexact_newton}
  and Table \ref{Example1c_exact_newton}, 
we observe that the number of Newton steps  is less than for Algorithm
\ref{inexat_newton_algorithm_for_multiple_goal_functionals}. 
However, the additional computational cost has to be considered, 
but we face 
a problem with nonlinear functionals, 
several reentrant corners and a very small regularization parameter
$\varepsilon = 10^{-10}$. 
Furthermore, these tables also suggest that we should compute 
both the primal and the adjoint error estimator to obtain a better approximation of the error.

\begin{table}[H]
	\tiny
	\centering
	\begin{tabular}{|r|r|r|c|c|c|c|}
		\hline
		$l$& \text{DOFs} & $I_{eff}$ & $\left|\frac{J_1(u)-J_1(u_h)}{J_1(u)}\right|$ & $\left|\frac{J_2(u)-J_2(u_h)}{J_2(u)}\right|$& $\left|\frac{J_3(u)-J_3(u_h)}{J_3(u)}\right|$& $\left|\frac{J_4(u)-J_4(u_h)}{J_4(u)}\right|$ \\ \hline 
		1   & 117    & 0.63      & {\color{black} 5.05E-02} & {\color{black} 3.02E-01 }& 1.10E-01 & 1.17E-01 \\ \hline
		2   & 161    & 0.53      & {\color{black} 1.53E-02} & 5.09E-02 & 4.94E-02 & {\color{black} 1.17E-01} \\ \hline
		3   & 290    & 0.84      & {\color{black} 8.25E-03} & 4.41E-02 & 2.14E-02 & {\color{black} 1.09E-01} \\ \hline
		4   & 447    & 0.81      & {\color{black} 4.86E-03} & {\color{black} 5.07E-02} & 1.53E-02 & 1.51E-02 \\ \hline
		5   & 791    & 0.96      & {\color{black} 2.09E-03} & {\color{black} 3.26E-02} & 1.12E-02 & 8.82E-03 \\ \hline
		6   & 1 331   & 1.14      & {\color{black} 1.37E-03} &{\color{black}  1.69E-02} & 8.44E-03 & 2.40E-03 \\ \hline
		7   & 2 541   & 1.92      & 1.65E-03 & 3.37E-03 & {\color{black} 4.38E-03} & {\color{black} 1.21E-03 }\\ \hline
		8   & 4 582   & 1.38      & {\color{black} 6.56E-04} &{\color{black}  3.78E-03} & 2.43E-03 & 8.57E-04 \\ \hline
		9   & 7 378   & 1.64      & 3.14E-04 & {\color{black} 2.52E-03} & 1.05E-03 &{\color{black} 2.06E-04} \\ \hline
		10   & 11 772  & 1.51      & {\color{black} 2.72E-04} &{\color{black}  1.73E-03} & 8.83E-04 & 3.51E-04 \\ \hline
		11  & 20 443  & 1.87      & 9.65E-05 &{\color{black}  5.65E-05} &{\color{black}  5.24E-04}& 5.84E-05 \\ \hline
		12  & 37 747  & 1.87      & {\color{black} 6.09E-05} &{\color{black}  3.05E-04} & 2.17E-04 & 1.20E-04 \\ \hline
		13  & 64 316  & 1.63      & {\color{black} 2.80E-05} & 1.30E-04 & {\color{black} 1.41E-04 }& 4.25E-05 \\ \hline
		14  & 104 832 & 1.44      & {\color{black} 1.04E-05} & {\color{black} 1.39E-04} & 7.18E-05 & 2.02E-05 \\ \hline
	\end{tabular}
		\caption{Section \ref{subsubsection Example 1c}, Case 1. Relative errors for the goal functionals on several refinement levels ($l$) and effectivity index $I_{eff}$ that is  computed for $J_{\mathfrak{E}}$ (\ref{ErrorrepresentationFunctionalapprox}). }
		\label{Example 1b: errortable}
\end{table}

	\begin{table}[H]
		\tiny
		\centering
		\begin{tabular}{|r|r|c|l|l|l|c|}
			\hline
				$l$ & \text{DOFs} & Error in $J_{\mathfrak{E}}$ & $I_{eff}$ & $I_{effp}$ & $I_{effa}$ & Newton steps \\ \hline
				1   & 117    & 7.43E-01                    & 0.63      & 0.6        & 0.65       & 8           \\ \hline
				2   & 161    & 2.54E-01                    & 0.53      & 0.27       & 0.79       & 2           \\ \hline
				3   & 290    & 1.94E-01                    & 0.84      & 0.24       & 1.43       & 2           \\ \hline
				4   & 447    & 8.40E-02                    & 0.81      & 0.28       & 1.34       & 5           \\ \hline
				5   & 791    & 5.39E-02                    & 0.96      & 0.48       & 1.45       & 4           \\ \hline
				6   & 1 331   & 2.89E-02                    & 1.14      & 0.24       & 2.05       & 1           \\ \hline
				7   & 2 541   & 1.06E-02                    & 1.92      & 0.02       & 3.86       & 3           \\ \hline
				8   & 4 582   & 7.71E-03                    & 1.38      & 0.41       & 2.36       & 4           \\ \hline
				9   & 7 378   & 4.09E-03                    & 1.64      & 0.74       & 2.55       & 2           \\ \hline
				10   & 11 772  & 3.23E-03                    & 1.51      & 0.7        & 2.32       & 4           \\ \hline
				11  & 20 443  & 6.23E-04                    & 1.87      & 1.06       & 2.67       & 2           \\ \hline
				12  & 37 747  & 7.03E-04                    & 1.87      & 0.66       & 3.07       & 6           \\ \hline
				13  & 64 316  & 3.41E-04                    & 1.63      & 0.39       & 2.87       & 4           \\ \hline
				14  & 104 832 & 2.42E-04                    & 1.44      & 0.5        & 2.38       & 3           \\ \hline
		\end{tabular}
			\captionof{table}{Errors in $J_\mathfrak{E}$, effectivity indices and number of Newton steps for $p=4$ and Algorithm \ref{inexat_newton_algorithm_for_multiple_goal_functionals}.}\label{Example1c_inexact_newton}
	\end{table}

\begin{figure}[H]
	\tiny
	\centering
	\begin{tabular}{|r|r|c|l|l|l|c|}
		\hline
		$l$ & \text{DOFs} & Error in $J_{\mathfrak{E}}$ & $I_{eff}$ & $I_{effp}$ & $I_{effa}$ & Newton steps \\ \hline
		1   & 117    & 7.43E-01       & 0.63      & 0.60       & 0.65       & 7           \\ \hline
		2   & 161    & 2.58E-01       & 0.52      & 0.26       & 0.79       & 4           \\ \hline
		3   & 290    & 1.94E-01       & 0.84      & 0.24       & 1.44       & 4           \\ \hline
		4   & 447    & 8.41E-02       & 0.81      & 0.28       & 1.34       & 6           \\ \hline
		5   & 791    & 5.40E-02       & 0.96      & 0.48       & 1.45       & 4           \\ \hline
		6   & 1 331   & 2.70E-02       & 1.38      & 0.38       & 2.39       & 6           \\ \hline
		7   & 2 198   & 2.02E-02       & 1.13      & 0.56       & 1.70       & 5           \\ \hline
		8   & 4 012   & 9.07E-03       & 1.43      & 0.70       & 2.16       & 6           \\ \hline
		9   & 6 879   & 4.02E-03       & 1.75      & 0.32       & 3.18       & 5           \\ \hline
		10   & 11 576  & 3.27E-03       & 1.40      & 0.62       & 2.19       & 6           \\ \hline
		11  & 20 187  & 8.20E-04       & 2.11      & 0.85       & 3.37       & 6           \\ \hline
		12  & 38 302  & 6.77E-04       & 1.78      & 0.54       & 3.02       & 7           \\ \hline
		13  & 64 740  & 3.12E-04       & 1.67      & 0.32       & 3.02       & 7           \\ \hline
		14  & 105 350 & 2.46E-04       & 1.35      & 0.47       & 2.22       & 5           \\ \hline
	\end{tabular}
	\captionof{table}{Errors in $J_\mathfrak{E}$, effectivity indices and
          number of Newton steps for $p=4$ and Algorithm \ref{inexat_newton_algorithm_for_multiple_goal_functionals} where $|\mathcal{A}(u^{l,k}_h)(z^{l,k}_h)|> 10^{-2} \eta_h^{l-1}$ is replaced by $\Vert \mathcal{A}(u_h^{l,k}) \Vert_{\ell_\infty}> 10^{-8}$.}\label{Example1c_exact_newton}
\end{figure}

	\begin{figure}
		\begin{minipage}[t]{0.45\linewidth}
			\centering
			\scalebox{0.65}{\input{Example1cCase1gnuplottex}}
			\captionof{figure}{Section \ref{subsubsection Example
                            1c}, Case 1. Error vs DOFs for $p=4$, $\varepsilon=10^{-10}$.}\label{Example1c: Case 1 gnuplot2}
		\end{minipage}
		\hspace{0.1 \linewidth}
		\begin{minipage}[t]{0.45\linewidth}
			\centering
			\scalebox{0.65}{\input{Example1cCase2gnuplottex}}
			\captionof{figure}{Section \ref{subsubsection Example 1c}, Case 2. Error vs DOFs for $p=1.33$ and $\varepsilon=10^{-10}$.}\label{Example1c: Case 2 gnuplot2}
		\end{minipage}
	\end{figure}
	
		\begin{figure}
			\begin{minipage}[t]{0.45\linewidth}
				\centering
				\scalebox{0.65}{\input{Example1cCase1gnuplot1tex}}
				\captionof{figure}{Section \ref{subsubsection
                                    Example 1c}, Case 1. Error vs \text{DOFs} for $p=4$ and $\varepsilon=10^{-10}$.}\label{Example1c: Case 1 gnuplot}
			\end{minipage}
			\hspace{0.1 \linewidth}
			\begin{minipage}[t]{0.45\linewidth}
				\centering
				\scalebox{0.65}{\input{Example1cCase2gnuplot1tex}}
				\captionof{figure}{Section \ref{subsubsection
                                    Example 1c}, Case 2. Error vs \text{DOFs} for $p=1.33$ and $\varepsilon=10^{-10}$.}\label{Example1c: Case 2 gnuplot}
			\end{minipage}
			%
			
		\end{figure}
		
\noindent
\textbf{Case 2 ($p=1.33$, $\varepsilon=10^{-10}$):}\\
We are interested in the same goal functional as in Case 1 but with $p=1.33$.
The following values, which are computed on a fine grid (8 global refinements,
$Q_c^2$ elements, $22\, 038\, 525$ \text{DOFs}) on the cluster RADON1,
are used to compute the reference values:
\newline
\begin{minipage}[t]{0.5\linewidth}
	\begin{align*}
		\int_{\Omega} u(x,y)\,d(x,y)\approx & 0.48510099008 \pm 4\times10^{-5},\\
		\int_{(2,3)\times(2,3)} u(x,y)\,d(x,y)\approx & 0.038058285978\pm 4\times 10^{-6},\\
		u(2.9,2.1)\approx & 0.034930138311\pm 4\times 10^{-6},
	\end{align*}
\end{minipage}
\begin{minipage}[t]{0.5\linewidth}
	\begin{align*}
		u(2.9,2.1)\approx & 0.034930138311\pm 4\times 10^{-6},\\
		u(0.6,0.6)\approx & 0.024478640536\pm 2\times 10^{-6},\\
		u(2.5,2.5)\approx & 0.039616834482\pm 4\times10^{-6}.\\
	\end{align*}
\end{minipage}

Considering again 
that 
the accuracy of the functional evaluations 
is valid, 
we observe that the relative error of $J_2$ is less than $8\times 10^{-4}$ and the relative  error of $J_1,J_3,J_4$ is less than $10^{-4}$.
As in Case 1, we compare the relative errors of the functionals in Figure \ref{Example1c: Case 2 gnuplot2}. 
Here we see that the error in $J_\mathfrak{E}$ bounds the relative errors. 
However, we loose control of the single functionals as long as they do not dominate the error, 
as for $J_2$ in Figure \ref{Example1c: Case 2 gnuplot2}.
In Case 2, $J_3$ and $J_1$, are these functionals. 
In the error plot given in  Figure \ref{Example1c: Case 2 gnuplot}, we  observe that the error  approximately behaves like 
$\mathcal{O}(\text{DOFs}^{-\frac{3}{4}})$ for a uniformly refined mesh, and  $\mathcal{O}(\text{DOFs}^{-1})$ for adaptive refinement, as for $p=4$. It turns out that the regions of refinement (except for corner singularities and the point evaluations)  have almost  a complementary structure for $p=1.33$ and $p=4$ as we can conclude from Figure \ref{Example 1c: p=4 Mesh11} and Figure \ref{Example 1c: p=1.33 Mesh11}.
\begin{figure}[H]
	\begin{minipage}[t]{0.3\linewidth}
		\hspace{-0.15\linewidth}
		\includegraphics[scale = 0.3]{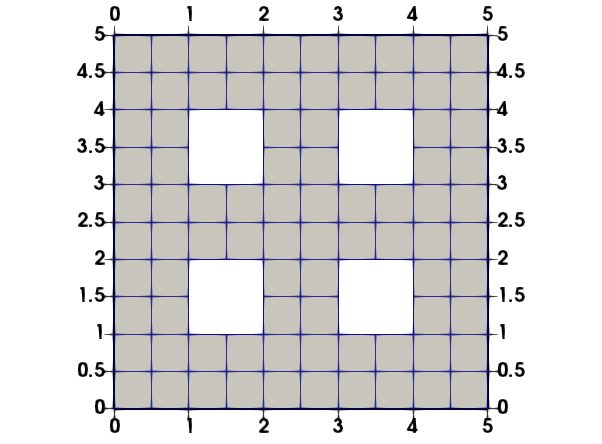}
		\caption{Initial mesh.}
		\label{initcheese}
	\end{minipage}
	\hspace{0.05 \linewidth}
		\begin{minipage}[t]{0.3\linewidth}

			\includegraphics[scale = 0.180]{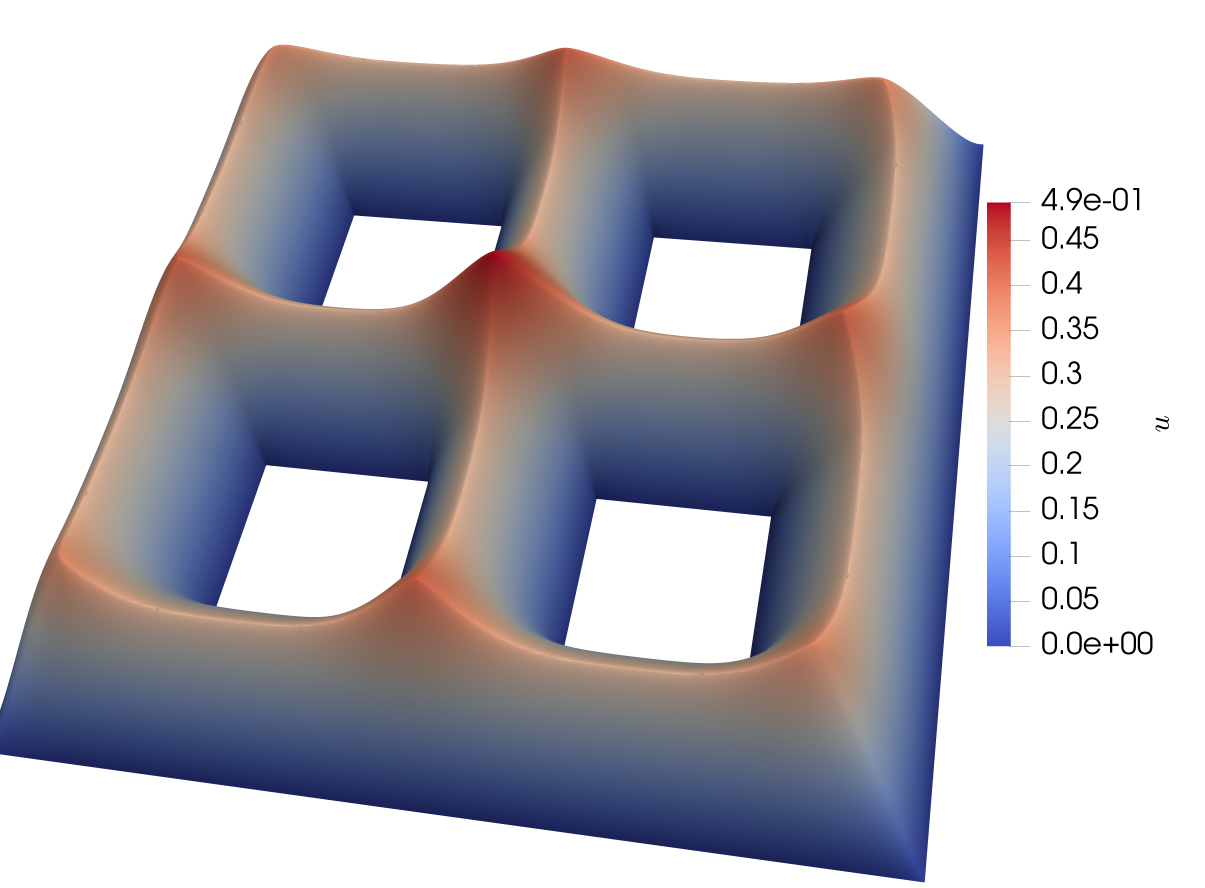}
			\caption{\ref{subsubsection Example 1c}: Solution for $p=4$.}
			\label{Example 1c: p=4 solution}
		\end{minipage}
		\hspace{0.05 \linewidth}
		\begin{minipage}[t]{0.3\linewidth}

			\includegraphics[scale = 0.180]{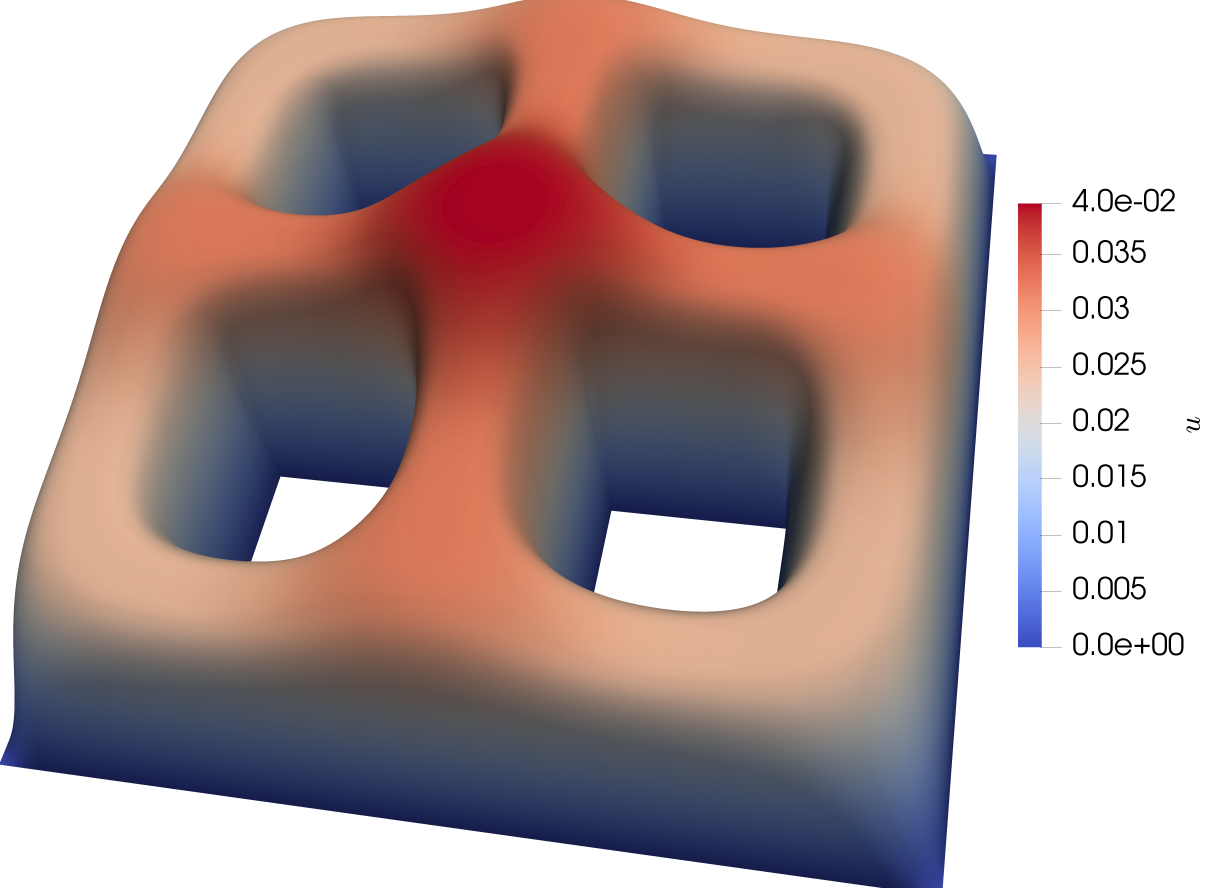}
			\caption{Section \ref{subsubsection Example 1c}: Solution for $p=1.33$.}
			\label{Example 1c: p=1.33 solution}
		\end{minipage}

\end{figure}
\begin{figure}[H]
	\begin{minipage}[t]{0.3\linewidth}
		\hspace{-0.35\linewidth}
		\includegraphics[scale = 0.48]{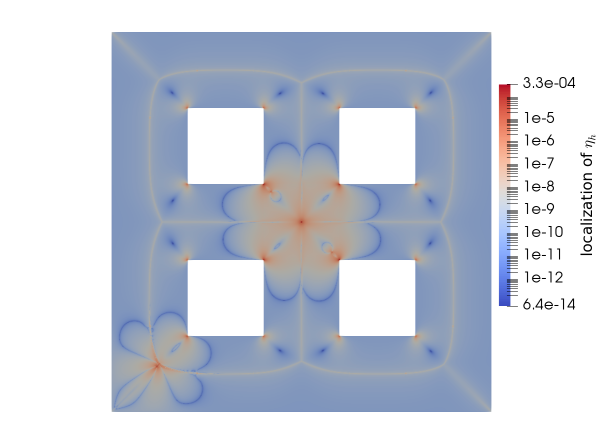}
		\caption{Section \ref{subsubsection Example 1c}: Local error estimator after $6$ uniform refinements for $p=4$.}
		\label{Example 1c:p4 error estimator}
	\end{minipage}
	\hspace{0.05 \linewidth}
	\begin{minipage}[t]{0.3\linewidth}
		\hspace{-0.25\linewidth}
		\includegraphics[scale = 0.18]{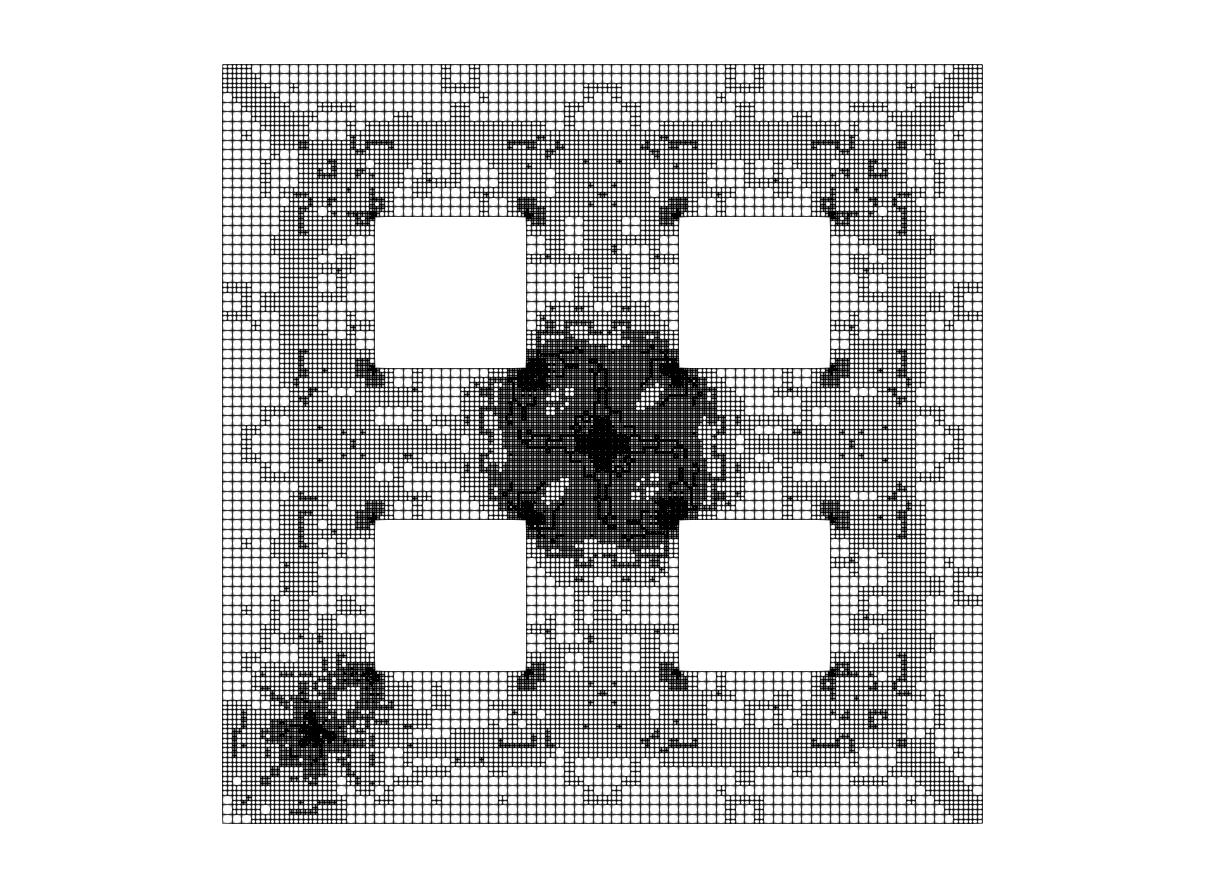}
		\caption{Section \ref{subsubsection Example 1c}: Mesh after $11$ adaptive refinements for $p=4$ ($37$ $747$ \text{DOFs}).}
		\label{Example 1c: p=4 Mesh11}
	\end{minipage}
	\hspace{0.05 \linewidth}
	\begin{minipage}[t]{0.3\linewidth}
		\hspace{-0.25\linewidth}
		\includegraphics[scale = 0.18]{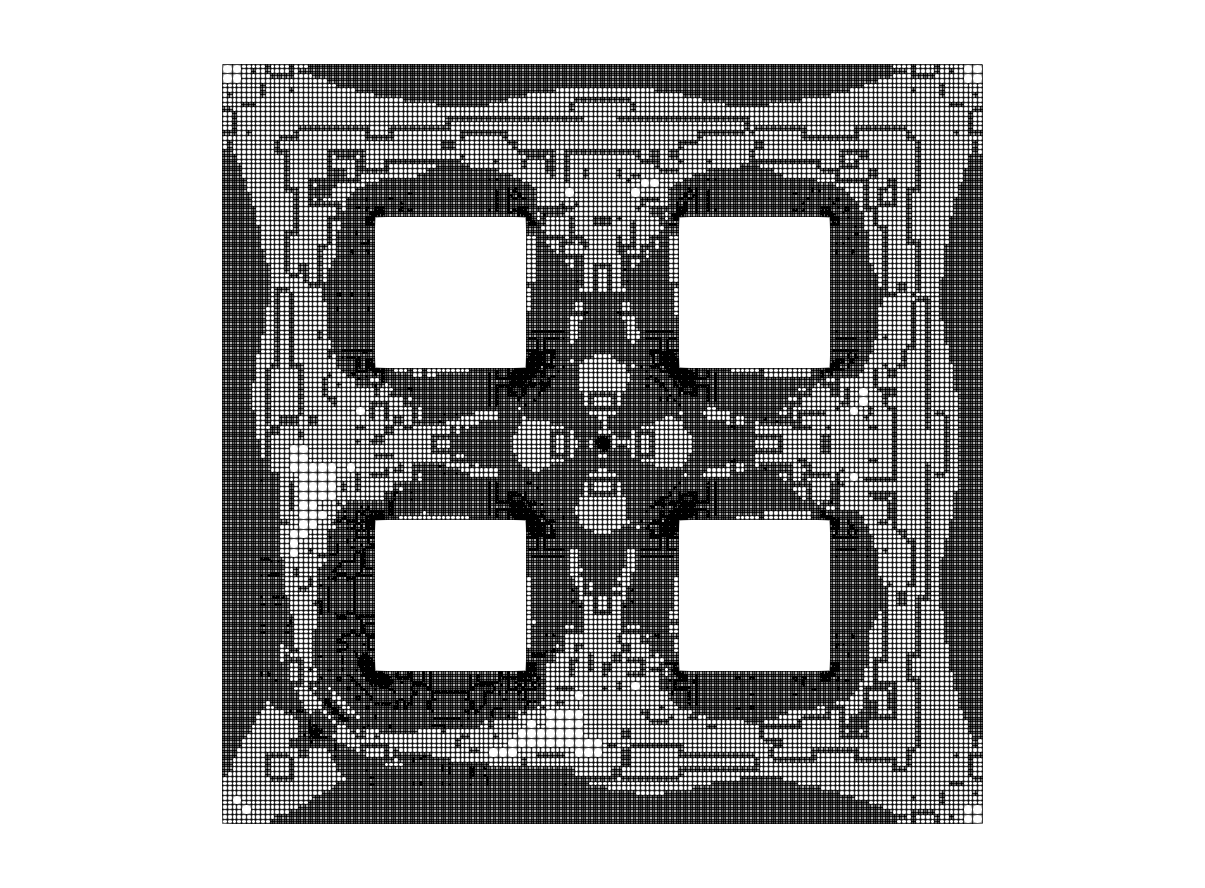}
		\caption{Section \ref{subsubsection Example 1c}: Mesh after $11$ adaptive refinements for $p=1.33$ ($40$ $499$ \text{DOFs}).}
		\label{Example 1c: p=1.33 Mesh11}
	\end{minipage}
\end{figure}



\subsection{Example 2: A quasilinear PDE system}
\label{QuasilinearvectorPDE}
%
In this second numerical test, we further substantiate our 
approach for a nonlinear, coupled, PDE system.
We consider the following nonlinear boundary value problem:
Find $u=(u_1,u_2,u_3) $ such that
\begin{align*}
	\begin{aligned}
		-\Delta u_1 +u_2 +u_3& =1,\quad \text{in } \Omega, \\
		-\Delta u_2 + g_1(1-u_2)-g_1(u_3)& =0,\quad \text{in } \Omega, \\
		-\text{div}(g_2(u_1+u_2)\nabla u_3) + g_1(u_3)-g_1(u_1) &= 0,\quad \text{in } \Omega, 
	\end{aligned} 
\end{align*}
is fulfilled in a weak sense, where 
\begin{align*}
u_1(x,y)=1-u_2(x,y)=u_3(x,y)& = \text{sign($y$)}\sqrt{\sqrt{x^2+y^2}-x} \qquad \text{on } \Gamma_D,\\
\nabla u_1.\vec{n} =\nabla  u_2.\vec{n} = g_2(u_1+u_2) \nabla u_3 .\vec{n} & =0  \qquad \text{on } \Gamma_N.
\end{align*}
Here sign denotes the signum function as defined in (\ref{sign}).
The functions $g_1$ and $g_2$ are given by
$	g_1(t):=e^t-\text{sin}(t-1) $ and $g_2(t):=e^{t^2-t}$, respectively.
Obviously a solution is given by $u_1(x,y)=1-u_2(x,y)=u_3(x,y) = \text{sign(y)}\sqrt{\sqrt{x^2+y^2}-x}$ in $\Omega$.
The computational domain is a slit domain as in \cite{EnWi17,Wi16_dwr_pff,AnMi15} 
and visualised in Figure \ref{Example 2: slit domain}.  
The boundary conditions above introduces  a discontinuity on the slit-boundary $(-1,0) \times \{0\}$ and consequently a discontinuity in the solution.  
The construction of this example was motivated by \cite{AnMi15,BoDa01}. 
Let $J_A,J_B,J_C,J_D,J_E,J_F$  be defined as follows:\\
\begin{minipage}[t]{0.33\linewidth}
\begin{align*}
J_A(u):=& u_3(-0.5,0.01),\\
J_D(u):=&\int_{\Omega} \Phi_D(x,y) \cdot u(x,y) \,d(x,y),
\end{align*}
\end{minipage}
\begin{minipage}[t]{0.33\linewidth}\textbf{}
	\begin{align*}
J_B(u)&:=u_1(-0.01,0.01),\\
J_E(u)&:=u_1(-0.9,-0.9),
	\end{align*}
\end{minipage}
\begin{minipage}[t]{0.33\linewidth}
	\begin{align*}
J_C(u):=&\int_{\Omega} \Phi_C(x,y) \cdot u(x,y)\, d(x,y),\\
J_F(u):= &u_2(-0.9,-0.1),\\
	\end{align*}
\end{minipage}
where $\Phi_C(x,y) :=(0,0, \chi_C(x,y) )$ and
$$
\Phi_D(x,y) :=(-4 \chi_D(x,y),\frac{2\chi_D(x,y)}{1-\text{sign(y)}\sqrt{\sqrt{x^2+y^2}-x}},4\chi_D(x,y)),$$
with $$\chi_C(x,y):=\begin{cases}
y-x  \quad &x<y \\
0 \quad &x \geq y 
\end{cases}
\quad \mbox{and} \quad
\chi_D(x,y):=\begin{cases}
	1  \quad &x,y>0  \\
	0  \quad &\text{else}
	\end{cases}
.$$

\vspace*{0.25cm}
We are now interested in the six goal functionals

\begin{minipage}[t]{0.33\linewidth}
	\begin{align*}
	J_{1}(u):=& J_B(u)J_D(u),\\
	J_{4}(u):=& J_B(u)J_E(u),\\
	\end{align*}
\end{minipage}
\begin{minipage}[t]{0.33\linewidth}
	\begin{align*}
	J_{2}(u):=& J_A(u)J_C(u),\\
	J_{6}(u):=& J_B^3(u)J_E(u),\\
	\end{align*}
\end{minipage}
\begin{minipage}[t]{0.33\linewidth}
	\begin{align*}
	J_{3}(u):=& J_A(u)J_C(u)J_F(u),\\
	J_{6}(u)=& J_C(u).\\
	\end{align*}
\end{minipage}
\newline
For the functional $J_B$, we can not expect optimal convergence rates for
uniform refinement due to the singularity at the slit tip.
Consequently, the same is true for the functionals
$J_1$, $ J_4$ and $ J_5$ as monitored in 
Figures~\ref{Example_2ErrorDOFsJ0J2},\ref{Example_2ErrorDOFsJ1J3} and \ref{Example_2ErrorDOFsJ4J5}. 
For uniform refinement, we got a relative error in $J_1$ of about $1.409531 \times
10^{-2}$ with $3$ $153$ $411$ \text{DOFs} as visualized in Figure~\ref{Example_2ErrorDOFsJ0J2}. 
To achieve a relative error less than
$1.409531 \times 10^{-2}$ our adaptive algorithm just needs $2$ $538 $ DOFs ( $1.042219 \times 10^{-2}$ ).  
If we use a similar number of DOFs ($3$ $021$ $045$), 
then  a relative error of $2.829422\times 10^{-6}$ is achieved.  
Figures~\ref{Example_2ErrorDOFsJ0J2}, \ref{Example_2ErrorDOFsJ1J3} and \ref{Example_2ErrorDOFsJ4J5}
might also lead to the conclusion that we obtain a convergence rate
$\mathcal{O}(\text{DOFs}^{-1})$ for all given functionals, where the functionals for
uniform refinement just converge with approximately
$\mathcal{O}(\text{DOFs}^{-\frac{1}{2}})$. 
This means, to obtain a relative error in
$J_1$ of about $2.829422\times 10^{-6}$ for uniform refinement, we would need
approximately $5 \times 10^{13}$ \text{DOFs}.
This would mean just storing the solution would require approximately 400 Terabyte. 
Therefore, obtaining this accuracy by means of uniform refinement would even be a hard task on the supercomputer Sunway TaihuLight\footnote{\url{http://www.nsccwx.cn/wxcyw/}}, which is  number one the of TOP500\footnote{\url{https://www.top500.org/lists/2017/11/}} list from November 2017.

We remark that $I_{eff}$, illustrated in Figure~\ref{Ieffs}, 
has no importance on course meshes since the approximations properties are bad anyway. 
On finer meshes, we see excellent behavior.

\begin{figure}[H]
	
	\begin{minipage}[t]{0.45\linewidth}
	\definecolor{ttttff}{rgb}{0.2,0.2,1}\definecolor{uuuuuu}{rgb}{0.26666666666666666,0.26666666666666666,0.26666666666666666}\definecolor{ffqqtt}{rgb}{1,0,0.2}\definecolor{cqcqcq}{rgb}{0.7529411764705882,0.7529411764705882,0.7529411764705882}\begin{tikzpicture}[line cap=round,line join=round,>=triangle 45,x=3.3cm,y=3.3cm]\draw [color=cqcqcq,dash pattern=on 1pt off 1pt, xstep=0.5cm,ystep=0.5cm] (-1.1859438838358034,-1.076009411190554) grid (1.1805682302792468,1.0396356943610634);\draw[->,color=black] (-1.1859438838358034,0) -- (1.1805682302792468,0);\foreach \x in {-1,-0.5,0.5,1}\draw[shift={(\x,0)},color=black] (0pt,2pt) -- (0pt,-2pt) node[below] {\footnotesize $\x$};\draw[->,color=black] (0,-1.076009411190554) -- (0,1.0396356943610634);\foreach \y in {-1,-0.5,0.5,1}\draw[shift={(0,\y)},color=black] (2pt,0pt) -- (-2pt,0pt) node[left] {\footnotesize $\y$};\draw[color=black] (0pt,-10pt) node[right] {\footnotesize $0$};\clip(-1.1859438838358034,-1.076009411190554) rectangle (1.1805682302792468,1.0396356943610634);
\fill[line width=2pt,color=ffqqtt,fill=ffqqtt,fill opacity=0.21] (-1,-1) -- (1,-1) -- (1,1) -- (-1,1) -- cycle;\draw [line width=2pt,color=ffqqtt] (-1,-1)-- (1,-1);\draw [line width=2pt,color=ffqqtt] (1,-1)-- (1,1);\draw [line width=2pt,color=ffqqtt] (1,1)-- (-1,1);\draw [line width=2pt,color=ffqqtt] (-1,1)-- (-1,-1);\draw [line width=3.2pt,color=ttttff] (-1,0)-- (0,0);\begin{scriptsize}\draw[color=ffqqtt] (0.10811176867057129,-0.9227001718708244) node {$\Gamma_D$};\draw [fill=uuuuuu] (0,0) circle (2pt);\draw[color=ttttff] (-0.3643544307905606,0.10857131538010847) node {$\Gamma_N$};\end{scriptsize}\end{tikzpicture}
		\caption{Example 2: The slit domain  $\Omega$ with $\Gamma_D$ (red) and $\Gamma_N$ (blue).}
		\label{Example 2: slit domain}
	\end{minipage}
	\hspace{0.1\linewidth}
	\begin{minipage}[t]{0.45\linewidth}
		\hspace{-0.15 \linewidth}
		\includegraphics[scale = 0.0575]{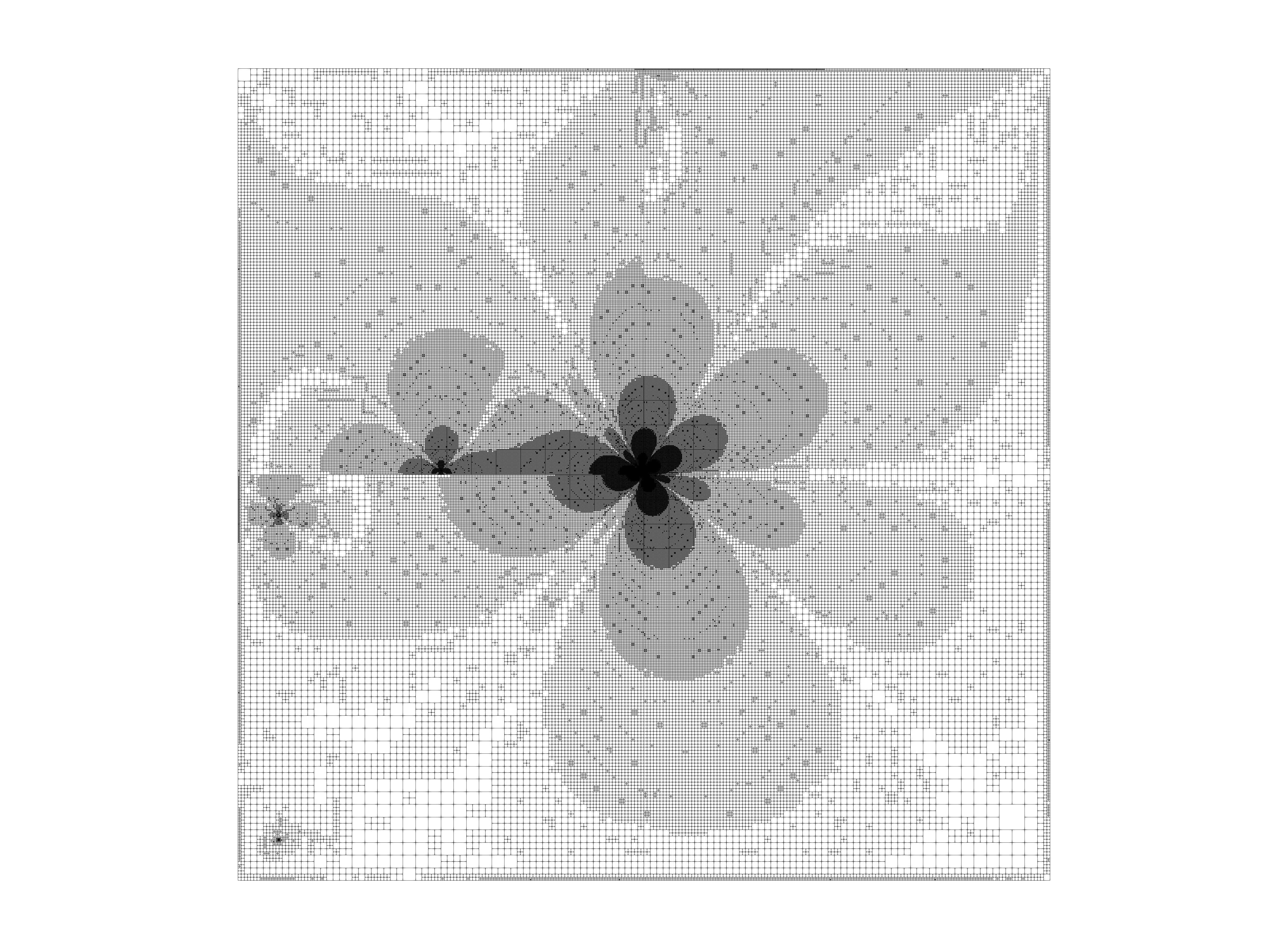}
		\caption{
		Example 2:  Adaptively refined mesh for $J_{\mathfrak{E}}$ after $24$ refinements ($683$ $118$ DOFs).} 
		\label{Example2Mesh24}
	\end{minipage}
\end{figure}

\begin{figure}[H]
	
	\begin{minipage}[t]{0.45\linewidth}
		\centering
		\scalebox{0.65}{\input{Error_nonlineartex}}
		\caption{
		Example 2: Error vs DOFs.}\label{Example_2ErrorDOFsJ1J3}
	\end{minipage}
	\hspace{0.1 \linewidth}
	\begin{minipage}[t]{0.45\linewidth}
		\centering
		\scalebox{0.65}{\input{Error_nonlinear2tex}}
		\caption{
		Example 2: Error vs DOFs.}\label{Example_2ErrorDOFsJ0J2}
	\end{minipage}

\end{figure}

\begin{figure}[H]
		\begin{minipage}[t]{0.45\linewidth}
			\centering
			\scalebox{0.65}{\input{Error_nonlinear3tex}}
			\caption{
			Example 2: Error vs DOFs.}\label{Example_2ErrorDOFsJ4J5}
		\end{minipage}
		\hspace{0.1 \linewidth}
	\begin{minipage}[t]{0.45\linewidth}
		\centering
		\scalebox{0.65}{\input{Error_nonlinear4tex}}
		\caption{
		        Example 2: Effectivity of the Error estimators.}
		        \label{Ieffs}
	\end{minipage}
\end{figure}
\section{Conclusions}
\label{sec_concl}

In this work, we have further developed adaptive 
schemes for multigoal-oriented a posteriori error estimation 
and mesh adaptivity. First, we extended the existing methods to 
nonlinear problems. Second, we combined the estimation of the 
discretization error with an estimation of the nonlinear iteration 
error in order to obtain adaptive stopping rules for Newton's method. In the key Sections 
\ref{Multigoalfunctionals} and \ref{sec_alg}, we formulated 
an abstract framework and its algorithmic realization. 
In Section \ref{sec_num_tests}, these developments were substantiated 
with several numerical tests. 
	Here, we studied the regularized 
$p$-Laplace problem and a nonlinear, coupled PDE system.
	Our findings demonstrate
	the performance of the algorithms
and specifically that the adjoint part of the error estimator, which is 
often neglected in the literature because of its higher computational cost, 
must be taken into account in order to achieve good effectivity indices.
In view of the geometric singularities, nonlinearities in both the PDE and the 
goal functionals, our results show excellent performance of our algorithms.

\section{Acknowledgments}
This work has been supported by the Austrian Science Fund (FWF) under the grant 
P 29181
`Goal-Oriented Error Control for Phase-Field Fracture Coupled to Multiphysics Problems'. 
The first author  thanks
the Doctoral Program on Computational Mathematics
at JKU Linz 
the Upper Austrian Goverment
for the support when starting the preparation of this work. 
The third author was supported 
by the Doctoral Program on Computational Mathematics 
during his visit at the Johannes Kepler University Linz in March 2018.

%
\bibliography{./lit}
\bibliographystyle{abbrv}
\end{document}